\documentclass{tran-l}

\usepackage{amssymb, bm, enumerate, mathtools}
\usepackage{mathrsfs}

\numberwithin{equation}{section}

\newtheorem{thm}{Theorem}[section]
\newtheorem{lem}[thm]{Lemma}
\newtheorem{cor}[thm]{Corollary}
\newtheorem{prop}[thm]{Proposition}

\newtheorem{defi}[thm]{Definition}

\theoremstyle{definition}

\theoremstyle{remark}
\newtheorem{remark}[thm]{Remark}

\newcommand{\ls}[2]{{\vphantom{#2}}^{#1\!}{#2}}

\newcommand{\mZ}{\mathbb{Z}}

\newcommand{\mQ}{\mathbb{Q}}

\newcommand{\mN}{\mathbb{N}}

\newcommand{\GL}{\operatorname{GL}}

\newcommand{\tr}{\operatorname{tr}}
\newcommand{\lda}{\lambda}

\newcommand{\bs}{\boldsymbol}
\newcommand{\scr}{\mathscr}

\newcommand{\ol}{\overline}

\newcommand{\Irr}{\operatorname{Irr}}
\newcommand{\al}{\alpha}
\newcommand{\be}{\beta}

\newcommand{\g}{\gamma}
\newcommand{\de}{\delta}
\newcommand{\ka}{\kappa}
\newcommand{\s}{\sigma}
\newcommand{\om}{\omega}

\newcommand{\Ind}{\operatorname{Ind}}
\newcommand{\Res}{\operatorname{Res}}

\newcommand{\Inf}{\operatorname{Inf}}

\newcommand{\lan}{\langle}
\newcommand{\ran}{\rangle}

\newcommand{\Lda}{\Lambda}

\newcommand{\witi}{\widetilde}
\newcommand{\Par}{\operatorname{Par}}
\newcommand{\Pow}{\operatorname{Pow}}
\newcommand{\CF}{\operatorname{CF}}
\newcommand{\cl}{\mathcal}
\newcommand{\Sym}{\mathrm{Sym}}

\newcommand{\PMap}{\operatorname{PMap}}

\newcommand{\wh}{\widehat}

\newcommand{\npa}{\medskip\noindent}

\newcommand{\pri}{\mathrm{pri}}
\newcommand{\diag}{\operatorname{diag}}
\newcommand{\ul}{\underline}
\newcommand{\Cart}{\operatorname{Cart}}
\newcommand{\vth}{\vartheta}
\newcommand{\ti}{\tilde}
\newcommand{\mtt}{\mathtt}
\newcommand{\G}{\Gamma}

\begin{document}

\title{Generalised Cartan invariants of symmetric groups}

\author{Anton Evseev}
\address{School of Mathematics, University of Birmingham, Edgbaston, Birmingham B15 2TT, UK}
\email{A.Evseev@bham.ac.uk}
\thanks{The author was supported by the EPSRC Postdoctoral Fellowship EP/G050244/1.}

 \begin{abstract} 
K{\"u}lshammer, Olsson, and Robinson developed an $\ell$-analogue of modular representation theory of 
symmetric groups where $\ell$ is not necessarily a prime. They gave a conjectural combinatorial 
description for invariant factors of the Cartan matrix in this context. We confirm their conjecture 
by proving a more precise blockwise version of the conjecture due to Bessenrodt and Hill. 
 \end{abstract}

\maketitle

\section{Introduction}\label{sec:intro}

Fundamental theory of representations of a finite group $G$ 
over an algebraically closed field of characteristic $\ell>0$ was developed by Brauer. 
An essential feature 
of $\ell$-modular representation theory is the construction of two sets of class functions 
defined on the elements of $G$ of order prime to $\ell$, 
namely, the irreducible Brauer characters and the 
projective indecomposable characters (see e.g.\ \cite[Chapter 2]{NavarroBook}).
These sets are dual to each other with respect to the usual scalar product. 
Further, there is a natural partition of each of these sets 
(as well as the set of ordinary irreducible characters of $G$)
into disjoint subsets that correspond to the \emph{$\ell$-blocks} of $G$. 
For the symmetric group $S_n$, K{\"u}lshammer, Olsson, and Robinson~\cite{KuelshammerOlssonRobinson2003} generalised character-theoretic aspects of Brauer's theory to the case when $\ell$ is not necessarily a prime and developed an analogue of block theory in this case. We begin by reviewing some of their definitions. 

For any finite group $G$, denote by $\Irr(G)$ the set of ordinary irreducible characters of $G$ and 
by $\cl C(G)$ the abelian group $\mZ[\Irr(G)]$ of virtual characters of $G$. Let $\ell,n\in \mN$. 
An element $g\in S_n$ is called \emph{$\ell$-singular} if the decomposition of $g$ into disjoint 
cycles includes at least one cycle of length 
divisible by $\ell$. 
Define
\[
 \scr P(S_n) = \{ \xi \in \cl C(G) \mid \xi(g) = 0 \text{ for all } \ell\text{-singular } g\in S_n \}.
\]
 Let $\{ \phi_t \}_{t\in T}$ be a $\mZ$-basis of $\scr P(S_n)$, indexed by a finite set $T$. 
The $\ell$-modular \emph{Cartan matrix}
of $S_n$ is the $T\times T$-matrix $\Cart_{\ell} (n) = (\lan \phi_t, \phi_{t'} \ran)_{t,t'\in T}$,
where $\lan \cdot,\cdot \ran$ is the usual scalar product of class functions. 
In this paper we are only concerned with the invariant factors of $\Cart_{\ell} (n)$. They do not depend on the choice of the basis. (If $\ell$ is prime, then projective indecomposable characters defined with
respect to $\ell$ form 
a basis of $\scr P(S_n)$.)

The set $\Irr(S_n)$ is parameterised by the partitions of $n$ in a standard way, 
and we write $s_{\lda}$ for the irreducible character corresponding to a partition $\lda$. 
If $\lda =(\lda_1,\ldots,\lda_t)$ is a partition (so that $\lda_1\ge\cdots\ge \lda_t>0$), 
we write $|\lda| = \sum_i \lda_i$ and $l(\lda) =t$. 

Let $\rho$ be a partition which is an $\ell$-core  
(see~\cite[\S 2.7]{JamesKerber1981}) and $e=|\rho|$.
We denote by $\Irr(S_n, \rho)$ the set of $s_{\lda} \in \Irr(S_n)$ such that $\rho$ is the $\ell$-core of
$\lda$. Then $\Irr(S_n,\rho)$ is the (combinatorial) \emph{$\ell$-block}, as defined in~\cite{KuelshammerOlssonRobinson2003}.
Write $\cl C(S_n,\rho)=\mZ[\Irr(S_n,\rho)]$ and 
$\scr P(S_n, \rho) = \scr P(S_n) \cap \cl C(S_n,\rho)$. 
It follows from~\cite[Corollary 4.3]{KuelshammerOlssonRobinson2003} that 
$\scr P(S_n) =\oplus_{\rho} \scr P(S_n, \rho)$,
where $\rho$ runs over all $\ell$-cores. 

Define the Cartan matrix $\Cart_{\ell} (S_n,\rho)$ to be the Gram matrix of a 
$\mZ$-basis of $\scr P(S_n,\rho)$ (that is, replace $\scr P(S_n)$ by $\scr P(S_n,\rho)$ in the definition
of $\Cart_{\ell} (n)$). Suppose that $n=e+\ell w$ for some $w\in \mZ_{\ge 0}$ (otherwise, $\cl C(S_n,\rho)=0$). 
The integer $w$ is called the \emph{weight} of the block in question. 
By~\cite[Theorem 6.1]{KuelshammerOlssonRobinson2003}, the invariant factors of $\Cart_{\ell}(n,\rho)$ 
depend only on $\ell$ and $w$. 

Let $\Par$ be the set of all partitions and $\Par(w)$ be the set of partitions of $w$.
Let $\lda = (\lda_1,\ldots,\lda_t)\in \Par$. 
If $j\in \mN$, denote by $m_j (\lda)$ the number of indices $i$ such that $\lda_i = j$. 
If $p$ is a prime, write $v_p (k)$ for the $p$-adic valuation of $k\in \mN$ and 
\begin{equation}\label{eq:dp}
 d_p (k) = v_p (k!) = \sum_{i=1}^{\infty} \left\lfloor \frac{k}{p^i} \right\rfloor.
\end{equation}
For $r\in \mZ_{\ge 0}$, define
\begin{equation}\label{eq:cpr}
  c_{p,r} (\lda) = \sum_{\substack{j\in \mN \\ 0\le v_p (j)<r }} 
\Big( (r-v_p (j)) m_j(\lda) + d_p (m_j (\lda)) \Big).
\end{equation}
If $\ell\in \mN$ and $\ell= \prod_i p_i^{r_i}$ is the prime factorisation of $\ell$, set
\begin{equation}\label{eq:vth}
 \vth_{\lda} (\ell) = \prod_i p_i^{c_{p,r_i}(\lda)}
\end{equation}
(see~\cite[Definition 3.5]{BessenrodtHill2010}).

 
Let $a,b\in \mZ_{\ge 0}$. 
Write $a^{\star b}$ for the sequence $a,\ldots,a$ with $b$ entries. 
Define $k(b,a)$ to be the number of tuples $(\lda^{(1)},\ldots, \lda^{(b)})$ of partitions such
that $\sum_{i=1}^b |\lda^{(i)}| = a$.  
If $R\subset R'$ are rings and $A$ and $B$ are $R'$-valued $a\times b$-matrices, then
$A$ and $B$ are said to be \emph{equivalent} over $R$ if there exist $U\in \GL_a (R)$ and $V\in \GL_b (R)$ such that $B= UAV$. (If the ring $R$ is not specified, it is assumed to be $\mZ$.)
The main aim of this paper is to prove the following result, conjectured by Bessenrodt and Hill
(see~\cite[Conjecture 5.3]{BessenrodtHill2010}). 

\begin{thm}\label{thm:BH}
 Let $\ell\ge 2$ and $w$ be integers. Let $\rho$ be an $\ell$-core and $n=|\rho|+\ell w$. Then 
the matrix $\Cart_{\ell} (n,\rho)$ is equivalent to the diagonal matrix with diagonal entries
\[
 \vth_{\ell} (\lda)^{\star k(\ell -2, w-|\lda|)},
\]
where $\lda$ runs over all partitions such that $|\lda|\le w$. 
\end{thm}

We note that the size of the diagonal matrix in Theorem~\ref{thm:BH} is $k(\ell -1,w)$. 

\begin{remark}
For prime numbers $\ell$, the elementary divisors of $\Cart_{\ell} (n,\rho)$ were determined 
by Olsson~\cite{Olsson1986}. Further, under the assumption that 
$r_i \le p_i$ for each $i$ in the above factorisation, Theorem~\ref{thm:BH} was proved 
in~\cite{BessenrodtHill2010} using results of~\cite{Hill2008}. 
Formulae for determinants of $\Cart_{\ell} (n,\rho)$ were given 
in~\cite{BessenrodtOlsson2003, BrundanKleshchev2002}.
\end{remark}

If $k\in \mN$ and $\pi$ is a set of primes, let $k_{\pi}$ be the greatest $a\in \mN$ such that
$a\mid k$ and all prime divisors of $a$ belong to $\pi$.  
Write $(\ell, k)$ for the greatest common divisor of $\ell$ and $k$, and let
$\pi(\ell, k)$ be the set of primes that divide $\ell/(\ell, k)$. 
For each $\lda\in \Par$, set
\[
 r_{\ell} (\lda) = \prod_{k=1}^{\infty} \left[
\left( \frac{\ell}{(\ell,k)} \right)^{\lfloor m_k (\lda)/\ell \rfloor}  \cdot
\left\lfloor \frac{m_k (\lda)}{\ell} \right\rfloor !_{\pi(\ell, k)} \right].
\]
The following corollary describes the invariant factors of $\Cart_{\ell} (n)$.
It was conjectured by K{\"u}lshammer, Olsson, and Robinson
(see~\cite[Conjecture 6.4]{KuelshammerOlssonRobinson2003}) and follows from Theorem~\ref{thm:BH} 
by~\cite[Theorem 5.2]{BessenrodtHill2010}.

\begin{cor}
 Let $\ell,n\in \mN$. The Cartan matrix $\Cart_{\ell} (n)$ is equivalent to the diagonal matrix 
with diagonal entries $r_{\ell} (\lda)$ where $\lda$ runs through the set of partitions 
$\lda = (\lda_1,\ldots,\lda_t)$ of $n$ such that $\ell \nmid \lda_i$ for all $i$. 
\end{cor}

\begin{remark}\label{rem:Donkin}
By a result of Donkin (see~\cite[Section 2, Remark 2]{Donkin2003}), 
the invariant factors of $\Cart_{\ell} (n)$ are the
same as those of the Cartan matrix of an Iwahori--Hecke algebra $\cl H_n (q)$ defined over any field where $q$ is a primitive $\ell$-th root of unity. In fact, 
Donkin's argument shows that $\Cart_{\ell} (n,\rho)$ is 
equivalent to the Cartan matrix of any block of weight $w$ in $\cl H_n (q)$. 
Thus, Theorem~\ref{thm:BH} gives a description of invariant factors of blocks of $\cl H_n (q)$. 
\end{remark}

The main step in the proof of Theorem~\ref{thm:BH} is to establish Theorem~\ref{thm:1t1p}, conjectured by Hill~\cite{Hill2008} (as well as 
Corollary~\ref{cor:1t1}, which follows from it). These results
describe the invariant factors of a certain 
$\Par(w)\times \Par(w)$-matrix, which is defined in Section~\ref{sec:scwr} and denoted by
$X = X^{(\mtt s,\mtt s)}_{\ell, w}$.

\begin{remark}
In~\cite[Theorem 1.1]{Hill2008}, Hill describes
the invariant factors of the Shapovalov form on the basic representation of any simply-laced affine 
Kac--Moody algebra in terms of the invariant factors of $X$. Thus, the proof of Theorem~\ref{thm:1t1p} completes a combinatorial description of the invariants of these Shapovalov forms. 
\end{remark}

Using results of Hill~\cite{Hill2008}, Bessenrodt and Hill~\cite{BessenrodtHill2010} proved that Theorem~\ref{thm:BH} is implied by 
Theorem~\ref{thm:1t1p}. Their reduction relies on the translation of the problem to Hecke algebras $\cl H_n (q)$ where $q$ is an $\ell$-th root of unity (see Remark~\ref{rem:Donkin}) and on 
results that relate the 
Grothendieck groups of finitely generated projective $\cl H_n (q)$-modules to the basic representation of the affine Kac--Moody algebra of type $A_{\ell -1}^{(1)}$
(see~\cite{Ariki1996},~\cite[Theorem 14.2]{Grojnowski1999} and~\cite[Chapter 9]{Kleshchev2005}). 
In Section~\ref{sec:scwr} we give a more direct and elementary proof of the reduction of
Theorem~\ref{thm:BH} to Theorem~\ref{thm:1t1p} that uses only character theory of symmetric groups and
wreath products. 
Our proof relies on an isometry constructed by Rouquier~\cite{Rouquier1994} between the block 
$\cl C(S_n, \rho)$ of Theorem~\ref{thm:BH} and 
the ``principal $\ell$-block'' of the wreath product $S_{\ell} \wr S_w$ 
 and on a result concerning class functions on wreath products proved in~\cite{Evseev2012}.

Intermediate results proved in Section~\ref{sec:scwr} show that certain matrices 
studied by Hill in~\cite{Hill2008} 
may be  interpreted naturally in terms of scalar products 
of class functions on $G\wr S_w$,
where $G$ is a finite group. 
These matrices are related to the inner product 
$\lan \cdot, \cdot \ran_{\ell}$ defined by Macdonald on the space of symmetric functions
  (see Remark~\ref{rem:Hill}).
The results of this paper determine the invariant factors of these matrices
(see Corollary~\ref{cor:Awrinv}).

Theorem~\ref{thm:1t1p} is proved in Sections~\ref{sec:powred} and~\ref{sec:fin}. 
In Section~\ref{sec:powred} we use Brauer's characterisation of characters to 
reduce Theorem~\ref{thm:1t1p} to the problem of finding the invariant
factors of a certain matrix $Y$ with rows and columns indexed only by the partitions
$\lda$ such that all 
parts $\lda_i$ are powers of a fixed prime $p$ (cf.\ the definitions before Theorem~\ref{thm:pow}). Finally, in Section~\ref{sec:fin}, we establish the invariant factors of $Y$ by a direct combinatorial argument and thereby complete the proof of 
Theorem~\ref{thm:BH}. 

\begin{remark}
 An important $\mZ$-grading on the Iwahori--Hecke algebras $\cl H_n (q)$ (and, more generally, on cyclotomic Hecke algebras of type $A$) was discovered by Brundan and Kleshchev~\cite{BrundanKleshchev2009a}. Recently, Ando, Suzuki, and Yamada~\cite{AndoSuzukiYamada2012} and Tsuchioka~\cite{Tsuchioka2012} have obtained formulae for determinants of \emph{graded} Cartan matrices and proposed conjectures concerning their invariant factors. In particular,~\cite[Conjecture 7.17]{Tsuchioka2012} generalises the statement of Theorem~\ref{thm:BH} (in the case when $\ell$ is a prime power). 
\end{remark}

\section{Notation and preliminaries}

In this section we introduce some general notation and review standard results that are used in the paper, 
in particular, those related to class functions on symmetric groups. 
Throughout, $\mZ_{\ge 0}$ and $\mN$ denote the sets of nonnegative and positive integers respectively.
If $a,b\in \mZ$, we write $[a,b]=\{ i\in \mZ \mid a\le i\le b\}$.

\npa \textbf{Matrices.}
Let $T$ and $Q$ be sets. If $A$ is a $T\times Q$-matrix, that is, a matrix 
with rows indexed by $T$ and columns indexed by $Q$, we write $A_{tq}$ for the $(t,q)$-entry of $A$. 
In Section~\ref{sec:scwr}, 
$A^n_{tq}$ denotes 
the $n$-th power of $A_{tq}$ (on the other hand, $(A^n)_{tq}$ is the $(t,q)$-entry of $A^n$).
All matrices considered will have only finitely many non-zero entries in each row and each column, so 
matrix multiplication is unambiguously defined even for infinite matrices. 
By $\diag\{(a_t)_{t\in T}\}$ we denote the diagonal $T\times T$-matrix with $(t,t)$-entry
equal to $a_t$ for each $t$. We write $A^{\tr}$ for the transpose of a matrix $A$. 
The identity $T\times T$-matrix
is denoted by $\mathbb I_T$. 

Let $R\subset R'$ be rings. As usual, $\GL_T(R)$ denotes the group of invertible $R$-valued $T\times T$-matrices $A$ such that $A^{-1}$ is $R$-valued. 
Two $R'$-valued $T\times Q$-matrices $A$ and $B$ are said to be 
\emph{row equivalent} over $R$ if there exists $U\in \GL_T (R)$ such that $B=UA$. 
The \emph{row space} of $A$ over $R$ is the $R$-span of the rows of $A$ as elements 
of $(R')^{Q}$, the free $R'$-module of vectors indexed by $Q$. 

\npa \textbf{Tuples and partitions.}
Let $T$ be a set and $w\in \mZ_{\ge 0}$. We define $I(T)$ to be the set of maps $j\colon T\to \mZ_{\ge 0}$ such that
$j(t) =0$ for all but finitely many $t\in T$. Further, $I_w (T)$ is the set of $j\in I(T)$
such that $\sum_t j(t) =w$. 

Suppose that $T$ is a finite set.
Denote by $\PMap(T)$ the set of all
maps from $T$ to $\Par$. If $\ul\lda\in \PMap(T)$, define 
$|\ul\lda| = \sum_{t\in T} |\ul\lda(t)|$. Set 
\[
\PMap_w(T) =\{ \ul\lda\in \PMap(T) \mid  |\ul\lda| = w \}.
\]
Note that $k(b,a) = |\PMap_a ([1,b])|$ for all $a,b\in \mZ_{\ge 0}$. 

The \emph{sum} of two partitions $\lda$ and $\mu$ is defined as the partition obtained by 
reordering the sequence $(\lda_1,\ldots,\lda_{l(\lda)},\mu_1,\ldots,\mu_{l(\mu)})$.
In particular, if 
$\ul\lda\in \PMap(T)$, then $m_j (\sum_{t\in T} \ul\lda(t)) = \sum_t m_j (\ul\lda(t))$ for all
$j\in \mN$. The sum of $n$ copies of $\lda$ is denoted by $\lda^{\star n}$. 

\npa \textbf{Class functions on symmetric groups.}
Let $\Lda = \oplus_{w\ge 0}\, \cl C(S_w)$. For any finite group $G$ write $\CF(G)$ for the 
set of $\mQ$-valued class functions on $G$. Then $\Lda_{\mQ} = \mQ \otimes_{\mZ} \Lda$ may be
identified with $\oplus_{w\ge 0} \CF(S_w)$. The scalar product $\lan \cdot, \cdot \ran$ 
on $\Lda_{\mQ}$ is defined via the standard scalar product on $\CF(S_w)$ in such a way that
the components $\CF(S_w)$ are orthogonal. 

By a \emph{graded basis} of $\Lda_{\mQ}$ we mean a $\mQ$-basis $\mtt u=(u_{\lda})_{\lda\in \Par}$ such 
that $(u_{\lda})_{\lda\in\Par(w)}$ is a basis of $\CF(S_w)$ for every $w$. 
If $\mtt u=(u_{\lda})$ and $\mtt v=(v_{\lda})$ are graded bases of $\Lda_{\mQ}$, we say that 
$(\mtt u,\mtt v)$ is a \emph{dual pair} if $\lan u_{\lda}, v_{\mu} \ran = \de_{\lda\mu}$ for all 
$\lda,\mu\in \Par$, where $\de_{\lda\mu}$ is the Kronecker delta. 

If $G, H$ are finite groups and $\phi\in \CF(G)$, $\psi\in \CF(H)$, then the outer tensor product $\phi\otimes \psi\in \CF(G\times H)$ is defined by $(\phi\otimes \psi)(g,h) =\phi(g) \psi(h)$. 
If $w=w_1+\cdots + w_n$ ($w_i\ge 0$), then
the direct product of the symmetric groups $S_{w_1},\ldots,S_{w_n}$
is viewed as a subgroup of $S_{w}$ (known as a \emph{Young subgroup}) in the usual way.
An element $f\in \Lda_{\mQ}$ is \emph{graded} if $f\in \CF(S_w)$ for some $w$. 
In this case we write $\deg(f)=w$. 
If $f$ and $f'$ are graded elements of $\Lda_{\mQ}$ of degrees $d$ and $w$ respectively, then their product is defined by 
\[
 ff' = \Ind_{S_d \times S_w}^{S_{d+w}} f\otimes f'.
\]
With this product, $\Lda_{\mQ}$ becomes a (graded) $\mQ$-algebra. 
The symbol $\Pi $, applied to elements of $\Lda_{\mQ}$, means this product.
When applied to sets or groups, $\Pi$ represents the usual direct product. 

By $A^{\times w}$ we mean the direct product of $w$ copies of a set or a group $A$. 
If $\phi$ is a class function on a group $G$, we write
$\phi^{\otimes w} = \phi\otimes \cdots \otimes \phi \in \CF(G^{\times w})$. 
If $U\le V$ are abelian groups, then $V^{\otimes w}$ is the tensor product (over $\mZ$) of $w$ copies
of $V$, and $U^{\otimes w}$ is viewed as a subgroup of $V^{\otimes w}$ in the obvious way. 

We will denote by $g_{\lda}$ an element of $S_{|\lda|}$ of cycle type $\lda\in \Par$.
We set 
\begin{equation}\label{eq:defz}
z_{\lda} = \prod_{i\in \mN} i^{m_i(\lda)} m_i (\lda)! = |C_{S_{|\lda|}} (g_{\lda})|.
\end{equation}
We will use graded bases $\mtt p=(p_{\lda})$, $\tilde{\mtt p}=(\tilde{p}_{\lda})$, $\mtt s=(s_{\lda})$,
$\mtt h=(h_{\lda})$ of $\Lda_{\mQ}$ defined as follows:
\begin{itemize}
 \item $p_{\lda} (g_{\mu}) = z_{\lda} \de_{\lda\mu}$ for all $\mu\in \Par(|\lda|)$;
 \item $\ti{p}_{\lda}(g_{\mu}) = \de_{\lda\mu}$, so that $\ti{p}_{\lda} = z^{-1}_{\lda} p_{\lda}$; 
 \item $s_{\lda}$ is the usual irreducible character of $S_{|\lda|}$ labelled by the partition $\lda$
(see~\cite[Eq.\ 2.3.8]{JamesKerber1981}), as in Section~\ref{sec:intro};
 \item $h_{n} = s_{(n)}$ and $h_{\lda} = h_{\lda_1} \cdots h_{\lda_t}$, 
where $\lda=(\lda_1,\ldots,\lda_t)$.
\end{itemize}
Note that $(\mtt p, \ti{\mtt p})$ and $(\mtt s, \mtt s)$ are dual pairs. 

While we find it convenient to use notation usually reserved for symmetric functions,
the elements just defined are to be viewed as class functions on symmetric groups, and our 
arguments are essentially character-theoretic. One may identify $\Lda$ with the 
ring of symmetric functions via the isomorphism of~\cite[\S I.7]{Macdonald1995}. 
With this identification, the elements $p_{\lda}$, $s_{\lda}$ and $h_{\lda}$ are the 
same as those defined in~\cite[\S I.2--3]{Macdonald1995}.

\section{Scalar products of class functions on wreath products}\label{sec:scwr}

We begin this section by summarising 
some notation and results concerning class functions on wreath products; for more detail, 
see~\cite[\S 2.3 and \S 4.1]{Evseev2012}. Let $G$ be a finite group and $w\in \mZ_{\ge 0}$. 
The wreath product $G\wr S_w$ consists of the tuples $(x_1,\ldots, x_w; \s)$ with
$x_i\in G$ and $\s\in S_w$. The group operation is defined by
\[
 (x_1,\ldots,x_w; \sigma) (y_1,\ldots, y_w; \tau) = (x_1 y_{\s^{-1}(1)}, \ldots, x_w y_{\s^{-1}(w)}; \s\tau),
\]
where we use the standard left action of $S_w$ on $[1,w]$. If $w=0$, then 
$G\wr S^w$ is the trivial group. 

By a \emph{cycle} in $S_w$ we understand either a non-identity cyclic permutation in $S_w$ 
or a $1$-cycle $(i)$ for some $i\in [1,w]$. Whenever $(i)$ is to be viewed as an element of $S_w$, 
it is interpreted as the identity element. The support of $(i)$ is defined as $\{ i \}$, while the
support of a non-identity cycle $\s$ is the set of points in $[1,w]$ moved by $\s$.
By $o(\s)$ we mean the order of a cycle $\s$, with the order of $(i)$ defined to be $1$.  
A tuple $\s_1,\ldots,\s_n$ is called a \emph{complete system} of cycles in $S_w$ 
if these cycles have disjoint supports and $\sum_i o(\s_i)=w$.

Whenever $\s$ is a cycle in $S_w$ and $x\in G$, we set
\[
 y_{\s} (x) = (1,\ldots,1,x,1,\ldots,1; \s) \in G\wr S_w,
\]
where $x$ appears in an entry belonging to the support of $\s$ (say, the first such entry). 
There is a unique equivalence relation on $G\wr S_w$ satisfying the following rule:
if $\s_1,\ldots,\s_n$ is a complete system of cycles, two elements of the form
$(u_1,\ldots,u_w; \tau)$ and $y_{\s_1}(x_1) \cdots y_{\s_n} (x_n)$ are 
equivalent if and only if $\tau = \s_1\cdots \s_n$ and 
$x_j = u_t u_{\s_j^{-1} (t)} \cdots u_{\s_j^{-(o(\s_j)-1)}}$ for all $j\in [1,n]$, where 
$t$ is the smallest element of the support of $\s_j$ (cf.~\cite[Eq.\ 4.2.1]{JamesKerber1981}).
Each equivalence class contains exactly one element of the form
$y_{\s_1} (x_1) \cdots y_{\s_n} (x_n)$ with $\s_1,\ldots,\s_n$ being a complete system,
and the equivalence class of such an element has size $|G|^{w-n}$. 
By~\cite[Theorem 4.2.8]{JamesKerber1981}, any two equivalent elements of $G\wr S_w$ are 
$G\wr S_w$-conjugate (even $G^{\times w}$-conjugate). By the same theorem,
if $\s_1,\ldots,\s_n$ is a complete system, two elements
$y_{\s_1} (x_1) \cdots y_{\s_n} (x_n)$ and $y_{\s_1}(u_1)\cdots y_{\s_n} (u_n)$
are $G\wr S_w$-conjugate if and only if there is a permutation $\tau$ of $[1,n]$ such that
$o(\s_j) = o(\s_{\tau j})$ and $x_j$ is $G$-conjugate to $u_{\tau j}$ for all $j\in [1,n]$. 

If $\phi\in \CF(G)$, we define $\phi^{\witi\otimes w} \in \CF(G\wr S_w)$ by setting
\[
 \phi^{\witi\otimes w} (y_{\s_1} (x_1) \cdots y_{\s_n} (x_n)) =
 \phi(x_1) \cdots \phi(x_n). 
\]
In the case when $\phi$ is a character afforded by a $\mQ G$-module, $\phi^{\witi\otimes w}$ 
is afforded by a corresponding $\mQ(G\wr S_w)$-module: see~\cite[Lemma 4.3.9]{JamesKerber1981}. 
Consider a tuple 
\begin{equation}\label{eq:Xi}
 \Xi = (  (\phi_1, f_1),\ldots, (\phi_n, f_n) )
\end{equation}
where $\phi_i \in \CF(G)$ and each $f_i$ is a graded element of $\Lda_{\mQ}$. 
Let $w_i = \deg(f_i)$ and suppose that $w=\sum_i w_i$. Then we define
\begin{equation}\label{eq:defzeta1}
 \zeta_{\Xi} = \Ind_{\prod_i (G\wr S_{w_i})}^{G\wr S_w}
 \bigotimes_{i=1}^n \left( \phi_i^{\witi\otimes w_i} \cdot \Inf_{S_{w_i}}^{G\wr S_{w_i}} f_i \right).
\end{equation}
Here, $\Inf_{{S_{w_i}}}^{G\wr S_{w_i}} f_i$ is the inflation of $f_i$, sending every
$g\in G\wr S_{w_i}$ to $f_i(g G^{\times w_i})$, and $\cdot$ is the inner tensor product:
$(f\cdot f') (g) = f(g)f'(g)$ for all $g$.
In the important special case when $\Xi = ((\phi,f))$ with $f\in \CF(S_w)$, we have
\[
 \zeta_{\Xi} = \zeta_{(\phi,f)} = \phi^{\witi{\otimes} w} \cdot \Inf_{S_w}^{G\wr S_w} f.
\]

Let $T$ be a finite set and $\phi\colon T\to \CF(G)$. 
For every $\ul\lda\in \PMap_w (T)$
 define $\zeta^{(\phi)}_{\ul\lda}$ to be equal to $\zeta_{\Xi}$ where 
\begin{equation}\label{eq:defzeta2}
 \Xi = ( ( \phi(t), s_{\ul\lda(t)} )  \mid t\in T ).
\end{equation}
If $T$ is a subset of $\CF(G)$ and $\phi$ is the identity map, we will write $\zeta_{\ul\lda}$
instead of $\zeta^{(\phi)}_{\ul\lda}$. 
These definitions are motivated, in part, by the fact that 
\begin{equation}\label{eq:wrirr}
 \Irr(G\wr S_w) = \{ \zeta_{\ul\lda} \mid \ul\lda \in \PMap_w (\Irr(G)) \}
\end{equation}
and the characters $\zeta_{\ul\lda}$ are distinct for different $\ul\lda\in \PMap_w (\Irr(G))$
(see~\cite[Theorem 4.3.34]{JamesKerber1981}).

For every $\lda\in \Par(w)$ and $\chi\in\CF(G\wr S_w)$ define 
$\om_{\lda} (\chi)\in \CF(G^{\times l(\lda)})$ by 
\[
 \om_{\lda}(\chi) (x_1,\ldots,x_n) = \chi (y_{\s_1} (x_1) \cdots y_{\s_n} (x_n) )
\]
where $n=l(\lda)$ and $\s_1,\ldots,\s_n$ form a complete system of cycles in $S_w$ with
$o(\s_i) = \lda_i$ for each $i$. We will view $\om_{\lda}(\chi)$ as an element of 
$\CF(G)^{\otimes l(\lda)}$.

Let $\cl X$ be a finitely generated 
subgroup of the abelian group $\CF(G)$. The subgroup $\cl X\wr S_w$ of
$\CF(G\wr S_w)$ is defined to be the $\mZ$-span of the class functions $\zeta_{\Xi}$ over all 
tuples $\Xi$ as in~\eqref{eq:Xi} such that $\phi_i\in \cl X$ and $f_i\in \Lda$ for all $i$.
A subgroup $U$ of a free abelian group $V$ is said to be \emph{pure} in $V$ 
if for every $v \in V$ such that $nv\in U$ for some $n\in \mZ-\{0\}$ we have $v\in U$.

\begin{thm}[\protect{\cite[Theorem 4.8 and Lemma 4.6]{Evseev2012}}]\label{thm:wrind}
 Let $\cl X$ be a pure subgroup of $\cl C(G)$. Then $\cl X\wr S_w$ is precisely 
the set of all $\xi\in \cl C(G\wr S_w)$ such that $\om_{\lda} (\xi) \in \cl X^{\otimes l(\lda)}$ for 
all $\lda\in \Par(w)$.  
\end{thm}

If $T$ is a finite set, let $I_w (T)$ denote the set of all maps $j\colon T\to \mZ_{\ge 0}$ such that
$\sum_{t\in T} j(t) = w$.

\begin{lem}\label{lem:Zbasis}
Let $\cl X$ be a finitely generated subgroup of the abelian group $\CF(G)$.
 Let $B$ be a $\mZ$-basis of $\cl X$. Then the class functions
$\zeta_{\ul\lda}$, $\ul\lda\in \PMap_w(B)$, form a $\mZ$-basis of $\cl X\wr S_w$. 
\end{lem}

\begin{proof}
First, we show that $\cl X\wr S_w$ is equal to the $\mZ$-span $V$ of the class functions $\zeta_{\ul\lda}$, 
$\ul\lda\in \PMap_w (B)$. We argue by induction on $w$. 
Consider a generator $\zeta_{\Xi}$ of $\cl X\wr S_w$, where $\Xi$ is as in~\eqref{eq:Xi} 
(with $\phi_i\in \cl X$ and $f_i\in \Lda$ for all $i$). We are to show that $\zeta_{\Xi}\in V$. 
By~\eqref{eq:defzeta1} and the inductive hypothesis, we immediately obtain $\zeta_{\Xi}\in \cl X\wr S_w$ unless $\deg(f_i)=0$ for all but one $i$. 
So we may assume that $\Xi = (\phi, f)$ for some $\phi\in \cl X$ and $f\in \Lda$. 
Write $\phi = \sum_{\psi\in B} n_{\psi} \psi$, where $n_{\psi}\in \mZ$. 
By~\cite[Lemma 2.5]{Evseev2012}, we have
\[
 \phi^{\witi\otimes w} = \sum_{j\in I_w (B)} \left( \prod_{\psi\in B} n_{\psi}^{j(\psi)} \right)
\Ind_{\prod_{\psi} G\wr S_{j(\psi)}}^{G\wr S_w} 
\left( \bigotimes_{\psi\in B} \, \psi^{\witi\otimes j(\psi)} \right).
\]
By the inductive hypothesis, the summand corresponding to $j$ lies in $V$ provided 
$j(\psi)<w$ for all $\psi\in B$. However, if $j(\psi)=w$ for some $\psi$, then the 
corresponding summand $n_{\psi}^w \psi^{\witi\otimes w}$ belongs to $V$ by definition. 
Hence, $\phi^{\witi\otimes w}\in V$, and it follows that 
$\zeta_{(\phi,f)} = \phi^{\witi\otimes w} \cdot \Inf_{S_w}^{G\wr S_w} f \in V$. 

Let $\cl X' = \mQ[\cl X]\cap \cl C(G)$, where $\mQ[\cl X]$ is the $\mQ$-span of $\cl X$ in $\CF(G)$. 
Then $\cl X'$ is a pure subgroup of $\cl C(G)$ and has dimension $|B|$. 
Let $B'$ be a $\mZ$-basis of $\cl X'$. 
Let $\lda=(\lda_1,\ldots,\lda_n)\in \Par(w)$ and $\psi_1,\ldots,\psi_n\in B$. 
Write $\bs\psi$ for the tuple $(\psi_1,\ldots,\psi_n)$.
Due to the above description of conjugacy classes of $G\wr S_w$, 
there exists a unique $\xi_{\lda,\bs\psi}\in \CF(G\wr S_w)$ such that 
$\om_{\lda} (\xi_{\lda,\bs\psi})= \psi_1\otimes \cdots \otimes \psi_n$ and
$\om_{\mu}(\xi_{\lda,\bs\psi})=0$ for all $\mu\ne \lda$. 
Clearly, the class functions $\xi_{\lda,\bm\psi}$ constructed in this way are 
linearly independent over $\mQ$. The number of such pairs $(\lda,\bs\psi)$ is $k(|B|,w)$. Indeed, 
a bijection from the set of these pairs onto $\PMap_w (B')$ is constructed as follows:
$(\lda,\bs\psi)\mapsto \ul\nu\in \PMap_w (B')$ where $\ul\nu(\psi)$ is the partition 
obtained by ordering the tuple $(\lda_i \mid \psi_i = \psi)$ (for each $\psi\in B'$). 
Let $(\lda,\bs\psi)$ be one of these pairs. Then 
$\om_{\mu} (\xi_{\lda,\bs\psi}) \in (\cl X')^{\otimes l(\mu)}$ for all $\mu\in \Par(w)$.
Further, $t\xi_{\lda,\bs\psi}\in \cl C(G\wr S_w)$ for some $t\in \mN$. Hence,
by Theorem~\ref{thm:wrind}, $t\xi_{\lda,\bs\psi} \in \cl X' \wr S_w$. 
Therefore, 
\[
\dim_{\mQ} (\mQ\otimes _{\mZ} (\cl X\wr S_w)) = 
\dim_{\mQ} (\mQ\otimes_{\mZ} (\cl X'\wr S_w)) \ge k(|B|,w) = |\PMap_w(B)|.
\]
The result follows. 
\end{proof}

\begin{remark}
Theorem~\ref{thm:wrind} is not really needed to prove linear independence in Lemma~\ref{lem:Zbasis}: 
for example, one 
can generalise the proofs of~\cite[Proposition 7.3 and Corollary 7.4]{Evseev2012}.
\end{remark}

Fix $\ell\in \mN$.
As in Section~\ref{sec:intro}, let $\rho$ be an $\ell$-core partition, with $|\rho|=e$. 
Let 
\[
\Irr_{\pri} (S_\ell)=\{ s_{(\ell-i,1^i)} \mid i\in [0,\ell-1]\}.
\]
(We write $1^i$ instead of $1^{\star i}$.) 
Write $\CF_{\pri}(S_\ell)$ for the $\mQ$-span of $\Irr_{\pri}(S_\ell)$. 
As in~\cite[Definition 3.3]{Evseev2012}, let 
\[
\Irr_{\pri}(S_\ell\wr S_w) = 
\{ \zeta_{\ul{\lda}} \mid \ul\lda \in \PMap_w (\Irr_{\pri} (S_\ell)) \}
\]
and $\cl C_{\pri}(S_\ell\wr S_w) = \mZ[\Irr_{\pri}(S_\ell\wr S_w)]$.
(If $\ell$ is a prime, then $\Irr_{\pri}(S_\ell\wr S_w)$ is the set of irreducible characters belonging 
to the principal $\ell$-block of $S_\ell\wr S_w$.)

Let $x\in S_\ell$ be an $\ell$-cycle. 
Define $\scr P_{\pri} (S_\ell\wr S_w)$ to be the set of all 
$\xi\in \cl C_{\pri} (S_\ell\wr S_w)$
such that
\begin{equation}\label{eq:xiproj}
 \xi(y_{\s_1} (x) y_{\s_2}(z_2) \cdots y_{\s_r} (z_r))=0
\end{equation}
whenever $\s_1,\ldots,\s_n$ is a complete system in $S_w$ and $z_2,\ldots,z_n\in S_\ell$. 
By~\cite[Th{\'e}or{\`e}me 2.11]{Rouquier1994}, 
there exists an  isomorphism  
$F\colon \cl C(S_{\ell w+e}, \rho) \to \cl C_{\pri}(S_\ell\wr S_w)$ of abelian groups
such that $F$ is an isometry. Moreover, we have
$F(\scr P(S_{\ell w+e},\rho)) = \scr P_{\pri}(S_\ell \wr S_w)$ due to the 
commutative diagram in the statement of~\cite[Th{\'e}or{\`e}me 2.11]{Rouquier1994}, 
if one interprets the vertical arrows of that diagram using Th{\'e}or{\`e}me 2.6 and Corollaire 2.10 of~\cite{Rouquier1994}.
(See also the proof of~\cite[Theorem 3.7]{Evseev2012}.)

Let $\xi\in \scr P_{\pri} (S_\ell\wr S_w)$ and $\mu\in \Par(w)$. Write $n=l(\mu)$. 
By~\eqref{eq:xiproj}, 
the class function $\om_{\mu} (\xi)$ belongs to the $\mQ$-vector space $V$ of all 
$\al\in \CF(S_\ell^{\times n})$ such that
$\al (z_1,\ldots,z_n)=0$ whenever at least one $z_i$ is an $\ell$-cycle. 
We have $\dim_{\mQ} V = j^n$
where $j$ is one less than the number of conjugacy classes in $S_\ell$. 
Since $j$ is the $\mZ$-rank of
$\scr P (S_\ell)$, the $\mZ$-rank of $\scr P (S_\ell)^{\otimes n}$ is $j^n$.
Clearly, $\scr P(S_\ell)^{\otimes n} \subset V$. Hence, $V$ is the $\mQ$-span of 
$\scr P(S_\ell)^{\otimes n}$.
Since $\scr P(S_\ell)$ is pure in $\cl C(S_\ell)$, we have 
$\scr P(S_\ell)^{\otimes n} = V\cap \cl C(S_\ell)^{\otimes n}$.
By~\cite[Lemma 4.11]{Evseev2012}, $\om_{\mu} (\xi) \in \cl C(S_\ell)^{\otimes n}$. 
Hence, $\om_{\mu} (\xi) \in \scr P (S_\ell)^{\otimes n}$.

Further, by Theorem~\ref{thm:wrind}, $\om_{\mu}(\xi) \in \cl C_{\pri}(S_\ell)^{\otimes n}$ for all $\mu$. 
Let $\cl X = \cl C_{\pri}(S_\ell)\cap \scr P(S_\ell)$. Since both $\cl C_{\pri} (S_\ell)$ and 
$\scr P(S_\ell)$ are pure in $\cl C(S_\ell)$, one easily sees that
\[
\cl X^{\otimes n} = \cl C_{\pri}(S_\ell)^{\otimes n} \cap \scr P(S_\ell)^{\otimes n} 
\ni \om_{\mu}(\xi).
\]
Therefore, by Theorem~\ref{thm:wrind}, $\scr P_{\pri}(S_\ell\wr S_w) = \cl X\wr S_w$.

For each $i\in [0,\ell-2]$ let
$\be_i = s_{(\ell-i,1^i)} + s_{(\ell-i-1,1^{i+1})}$. 
(When $\ell$ is prime, $\be_i$ are the projective indecomposable characters of
the principal block of $S_\ell$.)
By~\cite[Eq.\ 2.3.17]{JamesKerber1981}, $s_{(\ell-i,1^i)}(x)=(-1)^i$ for each $i$. 
Therefore, the set $\cl B = \{ \be_i \}_{i=0}^{l-2}$ is a $\mZ$-basis of $\cl X$. 
By Lemma~\ref{lem:Zbasis}, 
it follows that the set $\{ \zeta_{\ul{\lda}} \}_{\ul{\lda}\in \PMap_w (\cl B)}$ is a $\mZ$-basis of
$\cl X\wr S_w = \scr P_{\pri} (S_\ell\wr S_w)$. 
Since $F$ preserves scalar products and maps 
$\scr P(S_{\ell w+e}, \rho)$ onto $\cl X\wr S_w$, 
 we have proved the following result.

\begin{prop}\label{prop:isored}
 The matrix $\Cart_{\ell}(\ell w+e,\rho)$ is equivalent to the 
$\PMap_w(\cl B)\times \PMap_w(\cl B)$-matrix with $(\ul{\lda}, \ul{\mu})$-entry equal
to $\lan \zeta_{\ul{\lda}}, \zeta_{\ul{\mu}} \ran$.  
\end{prop}
%
%
%

Observe that the Gram matrix $(\lan \be_i,\be_j \ran)_{0\le i,j\le \ell-2}$ is 
\begin{equation}\label{eq:CartA}
 \begin{pmatrix}
  2 & 1 & 0 & \cdots & 0 & 0 \\
  1 & 2 & 1 & \cdots & 0 & 0 \\
  0 & 1 & 2 & \ddots & \vdots & 0 \\
  \vdots & \vdots & \ddots & \ddots & \ddots & \vdots \\
  0 & 0 & \cdots & \ddots & 2 & 1 \\
  0 & 0 & 0 & \cdots & 1 & 2
 \end{pmatrix}
\end{equation}
  %
After conjugation by the diagonal matrix with the $(i,i)$-entry equal to $(-1)^i$, this becomes
the classical Cartan matrix of type $A_{\ell-1}$. As is well known, the invariant factors of this matrix are $\ell,1,1,\ldots,1$ (with $1$ appearing $\ell-2$ times). 

This observation and Proposition~\ref{prop:isored} suggest the following general problem: given 
a finite set $T$ and a map $\phi\colon T\to \cl C(G)$ for a finite group $G$, describe
the invariant factors of the $\PMap_w (T)\times \PMap_w(T)$-matrix 
$(\lan \zeta_{\ul{\lda}}^{(\phi)}, \zeta_{\ul{\mu}}^{(\phi)} \ran)_{\ul{\lda}, \ul{\mu}}$ in terms of the
invariant factors of the $T\times T$-matrix $(\lan \phi(t), \phi(q) \ran)_{t,q\in T}$. 
In the case when $|T|=1$, the answer is given by Theorem~\ref{thm:1t1p}, which is 
proved in Sections~\ref{sec:powred} and~\ref{sec:fin}, and by Corollary~\ref{cor:1t1}. 
The rest of this section is devoted to an unsurprising reduction of the general problem to the case $|T|=1$ (see Corollary~\ref{cor:Awrinv}).

%
%
%
%

\begin{defi}\label{def:Awr}
 Let $\mtt u=(u_{\lda})$ and $\mtt v=(v_{\lda})$ be graded bases of $\Lda_{\mQ}$. 
Let $A$ be a $T\times Q$-matrix, where $T$ and $Q$ are finite sets. 
Then $A^{\wr}(\mtt u,\mtt v)$ is the $\PMap(T)\times \PMap(Q)$-matrix defined by 
\begin{flalign}
 (A^{\wr}(\mtt u,\mtt v))_{\ul{\lda}\ul{\mu}} =  \quad & \notag \\
 = \sum_{\ul{\nu} }
& \left(
\prod_{t} \left\lan u_{\ul{\lda}(t)},\textstyle\prod_q \tilde{p}_{\ul{\nu}(t,q)} \right\ran \cdot
\prod_q \left\lan v_{\ul{\mu}(q)}, \textstyle\prod_t p_{\ul{\nu}(t,q)} \right\ran \cdot
\displaystyle\prod_{t,q} A_{tq}^{l(\ul\nu(t,q))}   \right), \label{eq:defAwr}
\end{flalign}
where $\ul\nu$ runs through $\PMap(T\times Q)$ and $t,q$ run through $T,Q$ respectively. 
\end{defi}

Note that the summand indexed by $\ul{\nu}$ in the above formula is zero unless
$|\ul\lda| = |\ul\nu| = |\ul\mu|$. Write $A^{\wr w}(\mtt u,\mtt v)$ for the 
$\PMap_w (R)\times \PMap_w (T)$-submatrix of $A^{\wr}(\mtt u,\mtt v)$. Then $A^{\wr}(\mtt u,\mtt v)$ is block-diagonal, with
blocks equal to $A^{\wr w}(\mtt u,\mtt v)$, $w\ge 0$. 
The preceding definition is motivated by the following result.

\begin{lem}\label{lem:Awr}
 Let $\phi\colon T\to \CF(G)$ and $\psi\colon Q\to \CF(G)$ be arbitrary maps, 
where $T,Q$ are finite sets and $G$ is a finite group. 
Let $A = (\lan \phi(t), \psi(q) \ran)_{t\in T, \, q\in Q}$. Then for every $w\ge 0$ and  
$\ul\lda\in\PMap_w(T), \, \ul\mu\in \PMap_w (Q)$, we have 
\[ \lan \zeta^{(\phi)}_{\ul\lda}, \zeta^{(\psi)}_{\ul\mu}\ran = A^{\wr w} (\mtt s, \mtt s)_{\ul\lda \ul\mu}. \] 
\end{lem}

First, we prove a simpler lemma.

\begin{lem}\label{lem:wrsimple}
 Let $G$ be a finite group and $\phi,\psi\in \CF(G)$. If $\lda,\mu\in \Par(w)$, then
\[
\lan \zeta_{(\phi,p_{\lda})}, \zeta_{(\psi, \ti{p}_{\mu})} \ran = 
\de_{\lda\mu} \lan \phi, \psi \ran^{l(\lda)}.
\]  
\end{lem}

\begin{proof} The proof is similar to that of~\cite[Lemma 7.2]{Evseev2012}.
 Observe that $\zeta_{(\phi,p_{\lda})}$ vanishes outside the preimage in $G\wr S_w$ of the 
conjugacy class of $S_w$ consisting of the elements of cycle type $\lda$. A similar statement 
holds for $\zeta_{(\psi, \ti{p}_{\mu})}$, so the lemma holds if $\lda\ne \mu$. 
Assume that $\lda=\mu$ and fix a complete system of cycles $\s_1,\ldots,\s_n$ with orders 
$\lda_1,\ldots,\lda_n$ in $S_w$, where $n=l(\lda)$. 
With respect to the equivalence relation on $G\wr S_w$ 
described above, the equivalence class of an element of the form 
$y_{\s_1}(x_1)\cdots y_{\s_n}(x_n)$ contains exactly $|G|^{w-n}$
elements, which are all conjugate to $y_{\s_1} (x_1)\cdots y_{\s_n} (x_n)$. 
Also, $\s=\s_1\cdots\s_n$ has $w!/z_{\lda}$ conjugates in
$S_w$. Therefore,
\[
 \lan \zeta_{(\phi,p_{\lda})}, \zeta_{(\psi, \tilde{p}_{\lda} )} \ran = 
(z_{\lda}|G|^n)^{-1} p_{\lda}(\s) \ti{p}_{\lda}(\s) 
\sum_{x_1,\ldots,x_n\in G} \prod_{i=1}^n \phi(x_i) \psi(x_i^{-1})
= \lan \phi,\psi \ran^{n}. \qedhere
\]
\end{proof}

\begin{proof}[Proof of Lemma~\ref{lem:Awr}]
One may parameterise the double $\prod_{t} S_{|\ul{\lda}(t)|}$-$\prod_q S_{|\ul{\mu}(q)|}$-cosets in $S_{w}$ 
by the maps $j\in I_w (T\times Q)$ such that
\begin{equation}\label{eq:condI}
 \sum_q j(t,q) = |\ul{\lda}(t)| \text{ for all } t\in T \quad
\text{ and } \quad \sum_t j(t,q) = |\ul\mu(q)|  \text{ for all } q\in Q. 
\end{equation}
Here, as usual, the double coset containing $g\in S_w$ corresponds to the map $j$ defined by
$S_{|\ul\lda(t)|} \cap \ls{g}S_{|\ul\mu(q)|} \simeq S_{j(t,q)}$, where $S_{|\ul\lda(t)|}$ and 
$S_{|\ul\mu(q)|}$ are the 
appropriate direct factors of the two Young subgroups being considered. 
Using the definition of 
$\zeta_{\ul\lda}^{(\phi)}$ and $\zeta_{\ul\mu}^{(\psi)}$ (see~\eqref{eq:defzeta1}) and applying the Mackey formula,
we see that $\lan \zeta_{\ul\lda}^{(\phi)}, \zeta_{\ul\mu}^{(\psi)}  \ran = \sum_{j} a_{j}$ 
where the sum is over all $j\in I_w (T\times Q)$ satisfying~\eqref{eq:condI} and the summands are
\begin{equation}\label{eq:aj}
 a_j = \left\lan \Res^{\prod_t {L\wr S_{|\ul\lda(t)|}}}_{\prod_{t,q} L\wr S_{j(t,q)}} 
\bigotimes_t \zeta_{(\phi(t), s_{\ul\lda(t)} )}, \,
\Res^{\prod_q L\wr S_{|\ul\mu(q)|}}_{\prod_{t,q} L\wr S_{j(t,q)}}
\bigotimes_q \zeta_{(\psi(q), s_{\ul\mu(q)})} \right\ran. 
\end{equation}
Note that, whenever $D$ is a finite set, $i\in I_w(D)$, $\al\in \CF(G)$, and $f\in \Lda_{\mQ}$, we have
\begin{equation}\label{eq:reszeta}
 \Res^{L\wr S_w}_{\prod_{d\in D} L\wr S_{i(d)}} \zeta_{(\al, f)} 
= \zeta_{\left(\al, \, \Res^{S_w}_{\prod_d S_{i(d)}} f\right)}.
\end{equation}

Fix a map $j\in I_w (T\times Q)$ satisfying~\eqref{eq:condI}. 
For every $q\in Q$, 
\begin{equation}\label{eq:Awr2}
 \Res^{S_{|\ul\mu(q)|}}_{\prod_t S_{j(t,q)}} s_{\ul\mu(q)} = 
 \sum_{\ul\nu} \left( \lan s_{\ul\mu(q)}, \textstyle\prod_t p_{\ul\nu(t)} \ran \cdot
\bigotimes_t \tilde{p}_{\ul\nu(t)} \right),
\end{equation}
where the sum is over all $\ul\nu \in \PMap_{|\ul\mu(q)|} (T)$ such that $\ul\nu(t) = j(t,q)$ for all $t$. 
Indeed, $\lan s_{\ul\mu(q)}, \prod_t p_{\ul\nu (t)} \ran$ is the value of the character 
$s_{\ul\mu(q)}$ on an
element of cycle type $\sum_t \ul\nu(t)$.
Similarly, for every $t\in T$, 
\begin{equation}\label{eq:Awr3}
 \Res^{S_{|\ul\lda(t)|}}_{\prod_t S_{j(t,q)}} s_{\ul\lda(t)} =
\sum_{\ul\eta} \left(
\lan s_{\ul\lda(t)}, \textstyle\prod_q \tilde{p}_{\ul\eta(q)} \ran \cdot
\bigotimes_q p_{\ul\eta(q)} \right),
\end{equation}
where the sum is over all $\ul\eta\in \PMap_{|\lda(t)|} (T)$ 
such that $|\ul\eta(q)| = j(t,q)$ for all $q$.
After using~\eqref{eq:reszeta} and substituting~\eqref{eq:Awr2} and~\eqref{eq:Awr3}, 
Eq.~\eqref{eq:aj} becomes
\[
\begin{split}
 a_j &= \sum_{\ul\eta,\, \ul\nu} \left( 
\prod_{t} \lan s_{\ul\lda(t)}, \textstyle\prod_q \tilde{p}_{\ul\eta(t,q)} \ran \cdot
\displaystyle\prod_t \lan s_{\ul\mu(q)}, \textstyle\prod_t p_{\ul\nu(t,q)} \ran \cdot
\displaystyle\prod_{t,q} \lan \zeta_{(\phi(t),\, p_{\ul\eta(t,q)})}, 
\zeta_{(\psi(q),\, \tilde{p}_{\ul\nu(t,q)})}\ran \right) \\
&= \sum_{\ul\nu} \left(
\prod_{t} \lan s_{\ul\lda(t)}, \textstyle\prod_q \tilde{p}_{\ul\eta(t,q)} \ran \cdot
\displaystyle\prod_q \lan s_{\ul\mu(q)}, 
\textstyle\prod_t p_{\ul\nu(t,q)} \ran \cdot
\displaystyle\prod_{t,q} \lan \phi(t), \psi(q) \ran^{l(\ul\nu(t,q))} \right).
\end{split}
\]
Here $\ul\eta$ and $\ul\nu$ run through the set of elements of $\PMap_w(T\times Q)$ such that
$|\ul\eta(t,q)|=j(t,q)=|\ul\nu(t,q)|$ for all $t,q$, and the second equality holds by Lemma~\ref{lem:wrsimple}. 
 Summing over all $j$ satisfying~\eqref{eq:condI}, we obtain
\[
\begin{split}
 \lan \zeta_{\ul\lda}^{(\phi)}, \zeta_{\ul\mu}^{(\psi)} \ran &= \sum_j a_j \\
&=
\sum_{\ul\nu}  \left( 
\prod_{t} \lan s_{\ul\lda(t)}, \textstyle\prod_q \tilde{p}_{\ul\nu(t,q)} \ran \cdot
\displaystyle\prod_q \lan s_{\ul\mu(t)}, \textstyle\prod_t p_{\ul\nu(t,q)} \ran \cdot
\displaystyle\prod_{t,q} A_{tq}^{l(\ul\nu(t,q))}  \right),
\end{split}
\]
where $\ul\nu$ now runs through the elements of $\PMap_w (T\times Q)$ such that
$\sum_q |\ul\nu(t,q)| = |\ul\lda(t)|$ for all $t$ and $\sum_t |\ul\nu(t,q)| =|\ul\mu(q)|$ for all $q$.
Moreover, this formula remains true if we sum over \emph{all} $\ul\nu \in \PMap(T\times Q)$, as the extra summands 
are all equal to $0$. Comparing with Definition~\ref{def:Awr}, we deduce the result.
\end{proof}

\begin{remark}
Let $(\mtt u,\mtt v)$ be any dual pair of graded bases of $\Lda_{\mQ}$.
 Lemma~\ref{lem:Awr} remains true if one replaces $A^{\wr} (\mtt s, \mtt s)$ by $A^{\wr} (\mtt u,\mtt v)$ 
and replaces $s_{\lda}$ in the definitions of $\zeta_{\ul\lda}^{(\phi)}$ and $\zeta_{\ul\mu}^{(\psi)}$
(cf.~\eqref{eq:defzeta2}) by $u_{\lda}$ and $v_{\lda}$ respectively. 
\end{remark}

In the remainder of this section, $T,Q,Z$ denote arbitrary finite sets. 
Let $M$ be a $\Par\times \Par$-matrix. 
The $\PMap(T)\times \PMap(T)$-matrix $M^{\otimes T}$ is defined by
$(M^{\otimes T})_{\ul\lda\ul\mu} = \prod_{t\in T} M_{\ul\lda(t),\ul\mu(t)}$.
Thus, $M^{\otimes T}$ may be identified with the tensor product of $|T|$ copies of $M$.  
If $\mtt u=(u_{\lda})$ and $\mtt u'=(u'_{\lda})$ are graded bases of $\Lda$, the \emph{transition matrix} 
$M(\mtt u,\mtt u')$ is the $\Par\times\Par$-matrix defined by
the identity
\begin{equation}\label{eq:trans}
 u_{\lda} = \sum_{\mu\in \Par} 
M(\mtt u,\mtt u')_{\lda\mu} u'_{\mu} \quad\text{ for all }\lda\in \Par.
\end{equation} 
Let $M(\mtt u,\mtt u';w)$ be the $\Par(w)\times \Par(w)$-submatrix of $M(\mtt u,\mtt u')$.
Then $M(\mtt u,\mtt u')$ is block-diagonal with blocks $M (\mtt u,\mtt u'; w)$, $w\ge 0$. 

\begin{lem}\label{lem:conj}
 Let $A$ be a $T\times Q$-matrix. Suppose that $(\mtt u,\mtt v)$ and $(\mtt u',\mtt v')$ are dual pairs.
Let $M=M(\mtt u,\mtt u')$. Then
\[
 A^{\wr}(\mtt u,\mtt v) = M^{\otimes T} A^{\wr}(\mtt u',\mtt v') (M^{-1})^{\otimes Q}.
\]
\end{lem}

\begin{proof}
 Due to the duality conditions, we have $M(\mtt v,\mtt v') = (M^{\tr})^{-1}$. That is,
\begin{equation}\label{eq:conj1}
 v_{\lda} = \sum_{\mu\in\Par} (M^{-1})_{\mu\lda} v'_{\mu} \quad \text{ for all } \lda\in \Par.
\end{equation}
Substituting~\eqref{eq:trans} and~\eqref{eq:conj1} into~\eqref{eq:defAwr},
one obtains the result after a straightforward calculation. 
\end{proof}

\begin{remark}
 The remaining proofs of this section (except for those of Lemmas~\ref{lem:wrint} and~\ref{lem:idmat}) 
use essentially the same arguments as those presented in~\cite[Sections 3,4,6]{Hill2008} and~\cite[Section 3]{BessenrodtHill2010}, applied to a slightly more general situation. 
\end{remark}

Let $A$ be a $T\times Q$-matrix, where $T,Q$ are finite,  
and let $n\in \mZ_{\ge 0}$. Denote by $\lan T\ran$ 
a $\mQ$-vector space with
basis $T$. The $n$-th symmetric power $\Sym^n (\lan T\ran)$ has a basis that consists of the monomials $\prod_{t\in T} t^{i(t)}$ where $i$ runs through $I_n(T)$. 
It is easy to see that, with respect to this basis and the analogous basis
of $\Sym^n (\lan Q\ran)$, the matrix $\Sym^n (A)$ of the $n$-th symmetric power of the operator 
$A\colon \lan T\ran \to \lan Q \ran$ may be described as follows:
\begin{equation}\label{eq:sympower}
 \Sym^n (A)_{ij} = \sum_{f} \prod_{t\in T} \binom{i(t)}{ (f(t,q))_{q\in Q} } 
\prod_{t,q} A_{tq}^{f(t,q)} 
\end{equation}
 where the sum is over all $f\in I_n (T\times Q)$ 
such that $\sum_q f(t,q) = i(t)$ for all $t$ and 
$\sum_t f(t,q) = j(q)$ for all $q$. 
Here, $i\in I_n (T)$, $j\in I_n (Q)$ are arbitrary, and 
\[
 \binom{i(t)}{ (f(t,q))_{q\in Q} } = \frac{i(t)!}{\prod_{q\in Q} f(t,q)!}
\]
is the binomial coefficient.
Due to functoriality of symmetric powers, we have
\begin{equation}\label{eq:symmult}
 \Sym^n (AB) = \Sym^n (A) \Sym^n (B)
\end{equation}
whenever the product $AB$ of matrices is defined.

\begin{prop}\label{prop:mult}
Suppose that $(\mtt u, \mtt v)$ is a dual pair. 
 Let $A$ be 
a $T\times Q$-matrix and $B$ a $Q\times Z$-matrix. Then
\[
 (AB)^{\wr}(\mtt u,\mtt v) = A^{\wr} (\mtt u,\mtt v) B^{\wr}(\mtt u,\mtt v).
\]
\end{prop}

\begin{proof}
 We begin with the case when $\mtt u=\mtt p$ and $\mtt v=\ti{\mtt p}$. 
Note that, if $(\lda^i)_i$ is a tuple of partitions and $\al = \sum_i \lda^i$, then 
$
  p_{\al} = \prod_i p_{\lda^i}
$ (cf.~\cite[\S I.2]{Macdonald1995}). 
Also, recall that $\ti{p}_{\lda} = z^{-1}_{\lda} p_{\lda}$ for all $\lda$ and that $(\mtt p, \ti{\mtt p})$ is 
a dual pair. 
Using these facts and applying Definition~\ref{def:Awr}, we obtain
\begin{equation}\label{eq:mult1}
 A^{\wr} (\mtt p,\tilde{\mtt p})_{\ul\lda\ul\mu} = \sum_{\ul\nu} \left(
\prod_t z_{\ul\lda (t)} \cdot
\prod_{t,q}  z_{\ul\nu(t,q)}^{-1} A_{tq}^{l(\ul\nu(t,q))}  \right)
\end{equation}
for all $\ul\lda\in \PMap_w (T)$, $\ul\mu\in \PMap_w (Q)$, 
where the sum is over all $\ul\nu\in \PMap(T\times Q)$ such that $\sum_q \ul\nu(t,q) = \ul\lda(t)$ for all $t$ and $\sum_t \ul\nu(t,q) =\ul\mu(q)$ for all $q$. 

In particular, $A^{\wr}(\mtt p,\tilde{\mtt p}) =0$ unless 
$\sum_t \ul\lda(t) = \sum_q \ul\mu(q)$. 
So we have a block-diagonal decomposition of $A^{\wr} (\mtt p,\tilde{\mtt p})$, 
with blocks indexed by
maps $j\in I(\mN)$:
the block of $j$ is the intersection of the rows
indexed by the maps $\ul\lda\in \PMap(T)$ such that $\sum_{t} m_d(\ul\lda(t)) = j(d)$ 
for all $d\in \mN$ and 
the columns indexed by the maps 
$\ul\mu\in \PMap(Q)$ such that $\sum_q m_d(\ul\mu(q)) = j(d)$ for all $d$. 

If $E$ is any finite set and $\ul\al\in \PMap(E)$, define
$\wh{\ul\al} = (\ul{\wh\al}^d)_{d\in \mN} \in \prod_{d\in \mN} I(E)$
by $\ul{\wh\al}^d (e) = m_d (\ul\al(e))$ for all $d\in \mN$, $e\in E$ (cf.~\cite[Notation 3.2]{Hill2008}).
Fix $j\in I(\mN)$, and let $C^{(j)}$ be the corresponding block of $A^{\wr} (p,\tilde{p})$. 
The map $\ul\lda \mapsto \wh{\ul\lda}$ is a bijection from the set of rows of $C^{(j)}$ onto
$\prod_{d\in \mN} I_{j(d)} (T)$. Similarly, $\ul\mu \mapsto \wh{\ul\mu}$ is a bijection from the set
of columns of $C^{(j)}$ onto $\prod_{d\in \mN} I_{j(d)} (Q)$. 

Consider a row $\ul\lda$ and a column $\ul\mu$ of the block $j$.
 Let $\ul\nu\in \PMap (T\times Q)$, and write $i^{(d)} (t,q) = \ul{\wh\nu}^d (t,q)$ for all $d\in \mN$, $t\in T$, $q\in Q$. 
Observe that $\ul\nu$ satisfies 
 the conditions stated after Eq.~\eqref{eq:mult1} 
if and only if for each $d\in \mN$ we have 
$\sum_{q} i^{(d)} (t,q) = \ul{\wh\lda}^d (t)$ for all $t$ and $\sum_t i^{(d)} (t,q) = \ul{\wh\mu}^d (q)$
for all $q$. 
If these conditions are satisfied, then by~\eqref{eq:defz} we have
\[
 z_{\ul\lda(t)} \prod_q z_{\ul\nu(t,q)}^{-1} = 
\frac{\prod_{d\in \mN} m_d(\ul\lda(t))!}{\prod_{d\in\mN} \prod_q m_d (\ul\nu(q,t))!} =
\prod_{d\in \mN} \binom{\ul{\wh\lda}^d (t)}{(i^{(d)} (t,q))_{q\in Q} } \quad \text{ for all } t\in T.
\]
 Substituting this into~\eqref{eq:mult1}, we obtain 
\[
 C^{(j)}_{\ul\lda\ul\mu} = 
\prod_{d\in \mN} \left( \sum_{i^{(d)}} \prod_t \binom{\ul{\wh\lda}^d (t)}{(i^{(d)}(t,q))_{q\in Q} }
\prod_{t,q} A_{tq}^{i^{(d)} (t,q)} \right),
\]
where $i^{(d)}$ runs through the elements of $I_{j(d)} (T\times Q)$ 
satisfying the above conditions (for each $d\in \mN$).
Comparing this with~\eqref{eq:sympower}, we see that after the identifications 
$\ul\lda \mapsto \ul{\wh\lda}$ and $\ul\mu \mapsto \ul{\wh\mu}$ the block $C^{(j)}$ becomes equal to 
\[
\bigotimes_{d\in \mN} \Sym^{j(d)} (A).
\]
Due to~\eqref{eq:symmult}, we deduce that
\begin{equation}\label{eq:mult2}
 (AB)^{\wr} (\mtt p,\tilde{\mtt p}) = A^{\wr} (\mtt p,\tilde{\mtt p}) 
B^{\wr} (\mtt p,\tilde{\mtt p}). 
\end{equation}

Now consider the general case and let $M=M(\mtt u,\mtt p)$. Using Lemma~\ref{lem:conj} and Eq.~\eqref{eq:mult2}, we obtain 
\begin{align*}
 A^{\wr} (\mtt u,\mtt v) B^{\wr} (\mtt u,\mtt v) &= 
(M^{\otimes R} A^{\wr}(\mtt p,\ti{\mtt p}) (M^{-1})^{\otimes T}) 
(M^{\otimes T} B^{\wr} (\mtt p,\ti{\mtt p}) (M^{-1})^{\otimes Q}) \\
 &= M^{\otimes R} (AB)^{\wr} (\mtt p,\tilde{\mtt p}) (M^{-1})^{\otimes Q} 
 = (AB)^{\wr}(\mtt u,\mtt v). \qedhere
\end{align*} 
\end{proof}

\begin{lem}\label{lem:wrint}
 Suppose that $A$ is an integer $T\times Q$-matrix. Then the entries of $A^{\wr}(\mtt s, \mtt s)$ are integers. 
\end{lem}

\begin{proof}
 Let $G$ be the cyclic group of order $|T|$, and let $\phi\colon T\to \Irr(G)$ be an 
arbitrary bijection. For each $q\in Q$ set $\psi(q) = \sum_t A_{tq} \phi(q)$, so that
$A = (\lan \phi(t), \psi(q) \ran)_{q,t}$. By Lemma~\ref{lem:Awr}, the entries of $A^{\wr w}(\mtt s, \mtt s)$
are of the form $\lan \zeta_{\ul\lda}^{(\phi)}, \zeta_{\ul\mu}^{(\psi)} \ran$ where 
$\ul\lda \in \PMap_w(T)$ and $\ul\mu\in \PMap_w(Q)$. 
By~\cite[Lemma 2.6]{Evseev2012}, $\zeta_{\ul\lda}^{(\phi)}, \zeta_{\ul\mu}^{(\psi)}\in \cl C(G\wr S_w)$, so all entries of $A^{\wr w}(\mtt s, \mtt s)$ are integers. Since 
$A^{\wr}(\mtt s, \mtt s)$ is block-diagonal with blocks 
$A^{\wr w}(\mtt s, \mtt s)$, the result follows. 
\end{proof}

\begin{lem}\label{lem:idmat}
 We have $\mathbb I_T^{\wr} (\mtt s, \mtt s) = \mathbb I_{\PMap(T)}$.
\end{lem}

\begin{proof}
 Let $G$ be the cyclic group of order $|T|$ and $\phi\colon T\to \Irr(G)$ a bijection. 
Let $w\in \mZ_{\ge 0}$.
As we observed above (see~\eqref{eq:wrirr}), the functions $\zeta_{\ul\lda}^{(\phi)}$, 
$\ul\lda\in\PMap_w (T)$, 
are distinct irreducible characters of $G\wr S_w$. 
Hence, by Lemma~\ref{lem:Awr}, 
$A^{\wr w}{(\mtt s, \mtt s)}_{\ul\lda\ul\mu} = \lan \zeta_{\ul\lda}^{(\phi)}, \zeta_{\ul\mu}^{(\phi)} \ran = 
\de_{\ul\lda\ul\mu}$ for all $\ul\lda,\ul\mu\in \PMap_w (T)$. 
\end{proof}

\begin{prop}\label{prop:wrequiv}
 If $A$ and $B$ are equivalent $T\times Q$-matrices, then $A^{\wr w}(\mtt s, \mtt s)$ and $B^{\wr w}(\mtt s, \mtt s)$ are equivalent.  
\end{prop}

\begin{proof}
 The hypothesis means that there are matrices $M\in\GL_T (\mZ)$ and $N\in \GL_Q(\mZ)$ 
such that $MAN=B$. The matrices $M^{\wr w}(\mtt s, \mtt s)$, $(M^{-1})^{\wr w}(\mtt s, \mtt s)$,
$N^{\wr w}(\mtt s, \mtt s)$ and $(N^{-1})^{\wr w}(\mtt s, \mtt s)$ are integer-valued by Lemma~\ref{lem:wrint}. 
By Proposition~\ref{prop:mult} and Lemma~\ref{lem:idmat}, 
\[
M^{\wr w}(\mtt s, \mtt s) (M^{-1})^{\wr w}(\mtt s, \mtt s) = (\mathbb I_T)^{\wr w}(\mtt s, \mtt s) = \mathbb I_{\PMap_w (T)}.
\]
Thus, $M^{\wr w}(\mtt s, \mtt s)\in \GL_{\PMap_w (T)} (\mZ)$. Similarly, 
$N^{\wr w} (\mtt s, \mtt s) \in \GL_{\PMap_w (Q)} (\mZ)$. 
By Proposition~\ref{prop:mult}, 
\[
M^{\wr w}(\mtt s, \mtt s) A^{\wr w}(\mtt s, \mtt s) N^{\wr w}(\mtt s, \mtt s) = B^{\wr w} (\mtt s, \mtt s),
\]
and the result follows.
\end{proof}

Let $\ell \in \mZ$. Applying Definition~\ref{def:Awr} to the $1\times 1$-matrix $(\ell)$, 
set $X^{(\mtt u,\mtt v)}_{\ell, w} = (\ell)^{\wr w} (\mtt u,\mtt v)$ 
for any graded bases $\mtt u$ and $\mtt v$ of $\Lda_{\mQ}$. 
That is, $X^{(\mtt u,\mtt v)}_{\ell,w}$ is the $\Par(w)\times \Par(w)$-matrix given by
\begin{equation}\label{eq:descX}
 (X^{(\mtt u,\mtt v)}_{\ell,w})_{\lda\mu} = \sum_{\nu\in \Par(w)} \lan u_{\lda}, \tilde{p}_{\nu} \ran
\lan v_{\mu}, p_{\nu} \ran \ell^{l(\nu)}. 
\end{equation}
In particular, 
\begin{equation}\label{eq:Xpp}
X^{(\mtt p,\tilde{\mtt p})}_{\ell, w} = \diag\{(\ell^{l(\lda)})_{\lda\in \Par(w)}\}.
\end{equation}
By Lemma~\ref{lem:conj}, 
\begin{equation}\label{eq:Xconj}
 X^{(\mtt s, \mtt s)}_{\ell,w} = 
M(\mtt s,\mtt p; w) X^{(\mtt p,\tilde{\mtt p})}_{\ell,w} M(\mtt s,\mtt p; w)^{-1}.
\end{equation}
Therefore, 
the determinant of $X^{(\mtt s, \mtt s)}_{\ell,w}$ is a power of $\ell$ (cf.~\cite[Section 6]{Hill2008}).

In Sections~\ref{sec:powred} and~\ref{sec:fin} we will prove the following key result.

\begin{thm}\label{thm:1t1p}
Let $p$ be a prime and $r\ge 0$. Then the elementary divisors of $X^{(\mtt s, \mtt s)}_{p^r,w}$ are 
\[
p^{c_{p,r}(\lda)}, \quad \lda\in \Par(w).
\]
\end{thm}

Here, $c_{p,r} (\lda)$ are the integers defined by~\eqref{eq:cpr}.

\begin{remark}\label{rem:Hill}
 In~\cite[\S VI.10]{Macdonald1995} Macdonald defined a bilinear form 
$\lan \cdot, \cdot \ran_{\ell}$ on $\CF(S_w)$ (for each $\ell\in \mN$) by setting 
$\lan p_{\lda}, p_{\mu} \ran = \de_{\lda\mu} \ell^{l(\lda)} z_{\lda}$ for all $\lda,\mu\in \Par(w)$. By~\eqref{eq:Xpp} and~\eqref{eq:Xconj}, 
the invariant factors of this bilinear form restricted to
$\cl C(S_w)$ are the same as the invariant factors of $X^{(\mtt s,\mtt s)}_{\ell, w}$
(as $\{ s_{\lda} \}_{\lda\in \Par(w)}$ is a $\mZ$-basis of $\cl C(S_w)$).
Theorem~\ref{thm:1t1p}, stated in terms of the form $\lan \cdot, \cdot \ran_{p^r}$,
was proved by Hill for $r\le p$ and conjectured to hold in general: 
see~\cite[Theorem 1.3]{Hill2008}.
Our proof uses a different approach to that of Hill. In fact, the arguments of Section~\ref{sec:fin} become much simpler if $r$ is large (more precisely, if $p^r>w$):
see Remark~\ref{rem:simplecase}. 
\end{remark}

Recall that, for $\lda\in \Par$, the integer $\vth_{\lda} (\ell)$ is defined by~\eqref{eq:vth} for $\ell>0$, 
and set $\vth_{\lda} (0) =0$. 

\begin{cor}\label{cor:1t1}
 Let $\ell\in \mZ_{\ge 0}$. Then $X_{\ell,w}^{(\mtt s, \mtt s)}$ is equivalent to
 $\diag\{ (\vth_{\ell} (\lda))_{\lda\in \Par(w)} \}$. 
\end{cor}

\begin{proof} 
The result is clear for $\ell =0$, so assume that $\ell >0$. 
 Let $\ell =\prod_i p_i^{r_i}$ be the prime factorisation of $\ell$. 
Due to~\eqref{eq:Xpp} and~\eqref{eq:Xconj}, we have
$X^{(\mtt s, \mtt s)}_{\ell,w} = \prod_i X^{(\mtt s, \mtt s)}_{p_i^{r_i},w}$, where the product may be taken in any order. 
The result now follows from Theorem~\ref{thm:1t1p} and the Chinese Remainder 
Theorem: see~\cite[Section 6]{Hill2008} for details. 
\end{proof}

\begin{cor}\label{cor:Awrinv}
 Suppose that a $T\times T$-matrix $A$ is equivalent to $\diag\{(a_t)_{t\in T} \}$ for some 
$a_t\in \mZ_{\ge 0}$.
Then $A^{\wr w}(\mtt s, \mtt s)$ is equivalent to the diagonal matrix with diagonal entries
\[
  \prod_{t\in T} \vth_{\ul\lda(t)}(a_t), \quad \ul\lda \in \PMap_w (T).
\]
\end{cor}

\begin{proof}
 Due to Proposition~\ref{prop:wrequiv}, we may assume that $A=\diag\{ (a_t)_{t\in T} \}$. 
As $A_{tq}=0$ whenever $t\ne q$, Eq.~\eqref{eq:defAwr} becomes
\begin{equation}\label{eq:Awrinv1}
 A^{\wr} (\mtt s, \mtt s)_{\ul\lda\ul\mu} = \sum_{\ul\nu\in\Par} \prod_t
\left( \lan s_{\ul\lda(t)}, \tilde{p}_{\ul\nu(t)} \ran
\lan s_{\ul\mu(t)}, p_{\ul\nu(t)} \ran  a_t^{l(\ul\nu(t))} \right). 
\end{equation}
In particular, $A^{\wr}(\mtt s, \mtt s)$ is block-diagonal with blocks indexed by the maps
$j\in I_w (T)$, where a row or column indexed by $\ul\lda$ intersects the block of $j$ if 
and only if $|\ul\lda(t)| = j(t)$ for all $t$. Comparing~\eqref{eq:Awrinv1} with~\eqref{eq:descX},
we see that the block indexed by $j$ is exactly
\[
 \bigotimes_{t\in T} X_{a_t, j(t)}^{(\mtt s, \mtt s)}.
\]
The result now follows from Corollary~\ref{cor:1t1}, as invariant factors are well-behaved 
with respect to tensor products of matrices. 
\end{proof}

Theorem~\ref{thm:BH} may be deduced as follows. 
Consider 
$A=(\lan \be_i, \be_j \ran)_{0\le i,j\le \ell-2}$, a Cartan matrix of the principal $\ell$-block
of $S_\ell$ (see~\eqref{eq:CartA}), so that $A$ has invariant factors $\ell, 1,\ldots,1$. 
By Proposition~\ref{prop:isored} and Lemma~\ref{lem:Awr}, the matrix 
$\Cart_{\ell}(S_{\ell w+e}, \rho)$ is equivalent to
$A^{\wr w} (\mtt s,\mtt s)$. 
Note that $\vth_{\lda} (1) =1$ for all $\lda\in \Par$. 
Hence, by Corollary~\ref{cor:Awrinv}, the  
matrix $A^{\wr w}(\mtt s,\mtt s)$ is equivalent to 
\[
\diag\{ (\vth_{\ul\lda(0)} (\ell))_{\ul\lda\in \PMap_w ([0,\ell-2])} \}.
\]
Now for each $\lda\in \Par$ with $|\lda|\le w$, the number of maps 
$\ul\lda \in \PMap_w ([0,\ell-2])$ such that $\ul\lda(0)=\lda$ is equal to 
$|\PMap_{w-|\lda|} ([1,\ell -2])| = k(\ell -2, w-|\lda|)$. 
So $A^{\wr} (\mtt s,\mtt s)$ is equivalent to the diagonal matrix described in 
Theorem~\ref{thm:BH}. 
Thus, it remains only to prove Theorem~\ref{thm:1t1p}.

\section{Reduction to $p$-power partitions}\label{sec:powred}

From now on, we fix a prime $p$ and $r\in \mZ_{\ge 0}$. 
Also, we adopt the convention that diagonal matrices are denoted by lower-case letters. 
If $x=(x_{tq})_{t,q}$ is a diagonal matrix, we will write $x_{t}$ for $x_{tt}$. 
Let $w\ge 0$.
Define the diagonal $\Par(w)\times \Par(w)$-matrix $a=a^{(w)}$ by $a_{\lda} = p^{r l(\lda)}$, 
so that $a = X^{(\mtt p, \ti{\mtt p})}_{p^r,w}$ by~\eqref{eq:Xpp}. 

Let $M=M^{(w)}$ be the transition matrix $M(\mtt h,\ti{\mtt p}; w)$ and $X' = MaM^{-1}$. 
Due to~\eqref{eq:Xconj}, 
\[
X' = M(\mtt h,\mtt s; w) X^{(\mtt s,\mtt s)}_{p^r,w} M(\mtt h,\mtt s; w)^{-1}.
\]
It is well known that
$M(\mtt h,\mtt s; w)\in\GL_{\Par (w)} (\mZ)$ (see~\cite[Eq.\ 2.3.7]{JamesKerber1981}), so $X'$ is equivalent to $X^{(\tt s,\tt s)}_{p^r,w}$. 
(In fact, it is the matrix $X'$ rather than $X^{(\mtt s,\mtt s)}_{p^r,w}$ that is considered in~\cite{Hill2008}.) 

Let $\lda,\mu\in \Par(w)$. Define $\scr M_{\lda\mu}$ to be the set of all maps 
$f\colon [1, l(\mu)] \to [1,l(\lda)]$ such that $\sum_{j\in f^{-1} (i)} \mu_j =\lda_i$ for all
$i\in [1,l(\lda)]$. Since $h_{\lda}$ is the permutation character corresponding to the Young subgroup $\prod_i S_{\lda_i}$, we obtain
\begin{equation}\label{eq:M}
 M_{\lda\mu} = |\scr M_{\lda\mu}| \quad \text{ for all } \lda,\mu\in \Par(w)
\end{equation}
after applying the definition of induced character 
(alternatively, see~\cite[Statement I.6.9)]{Macdonald1995}).


As usual, let $\mZ_{(p)} = \{a/b \mid a,b\in \mZ, \, p\nmid b \}$, a subring of $\mQ$.  
Let $\ol{\mQ}$ be the algebraic closure of $\mQ$ and $\ol{\mZ}_{(p)}$ be the integral closure
of $\mZ_{(p)}$ in $\ol{\mQ}$.
For any finite group $G$, let $\CF(G; \ol{\mQ})$ be the abelian group of $\ol{\mQ}$-valued class functions
on $G$, and define the following subgroups of $\CF(G; \ol{\mQ})$:
$\cl C_{(p)}(G) = \mZ_{(p)}[\Irr(G)]$ and
$\ol{\cl C}_{(p)} (G) = \ol{\mZ}_{(p)} [\Irr(G)]$. 

\begin{remark}\label{rem:Zp}
We may work over $\mZ_{(p)}$ rather than $\mZ$ when proving Theorem~\ref{thm:1t1p}. 
Indeed, since $\det(X')$ is a power of $p$, any diagonal matrix with $p$-power diagonal entries which is equivalent to $X'$ over $\mZ_{(p)}$ must be equivalent to $X'$ (and hence to $X$) over $\mZ$. 
Thus, we may replace $M$ in the formula $X'=MaM^{-1}$ by any matrix $L$ which is 
row equivalent to $M$ over $\mZ_{(p)}$, that is, such that the rows of $L$ span
$\cl C_{(p)} (S_w)$ (in the sense that is made precise below). 
In this section we will use Brauer's characterisation of characters to find 
an especially nice matrix $L$ that satisfies this property; in particular, $L$ 
is block-diagonal with respect to a certain partition
of the set $\Par(w)$.   
This will considerably simplify the problem. 
\end{remark} 

Let $G$ be a finite group and $h\in G$ be a $p'$-element (that is, an element of order prime to $p$). 
Define a class function $\chi_{G,h} = \chi_{h}\in \CF(G; \ol{\mQ})$ as follows: 
\[
 \chi_h (g)= 
\begin{cases}
 1 & \text{if } g_{p'} \text{ is } G\text{-conjugate to } h, \\
 0 & \text{otherwise.}
\end{cases}
\]
Here, as usual, $g_{p'}$ denotes the $p'$-part of $g$ (that is,  $g_{p'}$ is a $p'$-element and 
 $g = g_p g_{p'} = g_{p'}g_{p}$ for some $p$-element $g_{p}\in G$).
Further, if $P$ is a subgroup of $C_G (h)$, define a map $\Res^{G,h}_P\colon \CF(G;\ol{\mQ}) \to \CF(P; \ol{\mQ})$ by
\[
 \Res^{G,h}_P (\xi)(x) = \xi(hx) \quad \text{ for all } \xi\in \CF(G, \ol{\mQ}), \, x\in P. 
\]

The following standard lemma, which is a consequence of Brauer's characterisation of characters, is key to the arguments of this section. 
(It will be used only in the case $G=S_w$.) 

\begin{lem}\label{lem:Br1}
\begin{enumerate}[(i)]
 \item\label{it:1}    $\chi_h\in \ol{\cl C}_{(p)} (G)$;
 \item\label{it:2}    
Let $\xi\in \CF (G; \ol{\mQ} )$. 
Suppose that for each $p'$-element $h\in G$ we have $\Res^{G,h}_P \xi \in \ol{\cl C}_{(p)} (P)$ where $P$ is a Sylow $p$-subgroup of $C_G (h)$.
Then $\xi\in \ol{\cl C}_{(p)} (G)$.
\end{enumerate}
\end{lem}

\begin{proof}
 Part (i) follows immediately from~\cite[Lemma 8.19]{IsaacsBook}. 

We note that the hypothesis of part (ii) for a given $p'$-element $h$ does not depend of the choice of $P$ because all Sylow $p$-subgroups of $C_G(h)$ are conjugate. 
By Brauer's characterisation of characters (see~\cite[Theorem 8.4]{IsaacsBook}), the conclusion of~\eqref{it:2} will follow once we show that
$\Res^{G}_{E} (\xi) \in \ol{\cl C}_{(p)} (E)$ for all elementary subgroups $E$ of $G$. 
For every such $E$ we have $E=Q\times P$ where $P$ is a $p$-group and $Q$ is a $p'$-group. Let $\cl Q$ be a set of representatives of the conjugacy classes of $Q$.
Then
\[
 \Res^G_E (\xi) = \sum_{q\in \cl Q} (\chi_{Q,q} \otimes \Res^{G,q}_P (\xi) ).
\]
By the hypothesis, $\Res^{G,q}_P (\xi) \in \ol{\cl C}_{(p)} (P)$ for all $q\in Q$. Also, $\chi_{Q,q}\in \ol{\cl C}_{(p)} (Q)$ by~\eqref{it:1}. 
The result follows.
\end{proof}


\begin{defi}
Let $T$ be a finite set.
Let $R$ be an integral domain with field of fractions $K$. 
Denote by $K^T$ the vector space of row vectors $v=(v_t)_{t\in T}$ with $v_t\in K$. 
If $Q$ is a subset of $T$, define $\pi_Q \colon K^T \to K^T$ by 
\[
 \pi_Q (v)_t = \begin{cases}
                v_t & \text{if } t\in Q, \\
		0  & \text{if } t\notin Q. 
               \end{cases}
\qquad \text{(for } t\in T\text{)}
\]
Let $T = \sqcup_i T_i$ be a set partition of $T$ and
$A$ be a finite $T\times T$-matrix with entries in $K$. 
Let $V \subset K^T$ be the row space of $A$ over $R$. 
We say that $A$ \emph{splits} over $R$ with respect to the given set partition of $T$ if
$\pi_{T_i} (V)\subset V$ for all $i$.
\end{defi}

We use this definition in the case $T=\Par(w)$ as follows. Let $\mN_{p'}$ be the set of all natural numbers that are prime to $p$. 
Denote by $\Par'(w)$ the set of all partitions $\nu\in \Par(w)$ such that $\nu_i \in \mN_{p'}$ for all $i$
(such $\nu$ are called \emph{$p$-class regular} in~\cite{KuelshammerOlssonRobinson2003}). 
Let $\nu\in \Par'(w)$. Recall that $g_{\nu}\in S_{w}$ is a fixed element of cycle type $\nu$.
For each $\nu\in\Par'(w)$ 
define 
$\Par(w, \nu)$ to be the set of all $\lda\in \Par(w)$ such that
$\sum_{n\ge 0} m_{jp^n} (\lda) p^n = m_j (\nu)$ for all $j\in \mN_{p'}$. 
This leads to the set partition
\begin{equation}\label{eq:parpart}
 \Par(w) = \bigsqcup_{\nu\in\Par'(w)} \Par(w, \nu). 
\end{equation}
Note that an element $g\in S_w$ has cycle type belonging to $\Par (w,\nu)$ if and only if
$g_{p'}$ has cycle type $\nu$. We will identify $\mQ^{\Par(w)}$ with $\CF(S_w)$ via 
\begin{equation}\label{eq:ident}
  v\mapsto \sum_{\lda\in\Par(w)} v_{\lda} \ti{p}_{\lda}. 
\end{equation}
With this identification, $\cl C(S_w)$ is the row space of the character table 
$M(\mtt s, \ti{\mtt p})$.
The row space of $M = M(\mtt h, \ti{\mtt p}; w)$ also equals $\cl C(S_w)$ since 
$M(\mtt h, \mtt s; w)\in\GL_{\Par(w)} (\mZ)$. 

Let $\xi\in \cl C(S_w)$ be the character corresponding to a row of $M$. Then, for every $\nu\in \Par'(w)$, the class function
$\chi_{g_{\nu}} \cdot \xi$ corresponds to $\pi_{\Par(w,\nu)} (\xi)$. 
However, by Lemma~\ref{lem:Br1}\eqref{it:1}, we have $\chi_{g_{\nu}} \cdot \xi \in \ol{\cl C}_{(p)}(S_w)$.
This means that $M$ splits over $\ol{\mZ}_{(p)}$ with respect to the set partition~\eqref{eq:parpart}. 
Since $M$ is rational-valued, we deduce the following more precise result using standard ring theory. 

\begin{prop}\label{prop:Msplit}
 The matrix $M$ splits over $\mZ_{(p)}$ with respect to the set partition 
$\Par(w) = \bigsqcup_{\nu\in\Par'(w)} \Par(w, \nu)$. 
\end{prop}

We will use the following general result on split matrices.

\begin{lem}\label{lem:tr}
Let $R$ be an integral domain with field of fractions $K$. Suppose that $T=\sqcup_i T_i$,
where $T$ is a finite set. Let $A$ be a $T\times T$-matrix with entries in $K$ that 
splits over $R$ with respect to this set partition. Suppose that $A$ is lower-triangular 
with respect to some total order on $T$ and that $A_{tt}\ne 0$ for all $t\in T$.
Define the $T\times T$-matrix $\witi{A}$ by
\[
 \witi{A}_{tq} = 
\begin{cases}
  A_{tq} & \text{if } t,q\in T_i \text{ for some } i, \\
 0 & \text{otherwise}.
\end{cases}
\qquad \qquad (t,q\in T)
\]
Then $A$ is row equivalent to $\witi{A}$ over $R$. 
\end{lem}

\begin{proof}
Let $<$ be the given total order on $T$, and
write $A_t$ for row $t$ of $A$. It is enough to prove the following: if $t\in T_i$ and $j\ne i$, then
\begin{equation}\label{eq:pit}
\pi_{T_j} (A_t) = \sum_{\substack{u\in T \\ u<t}} \al_u A_u \quad \text{ for some coefficients } \al_u \in R. 
\end{equation}
Indeed, applying~\eqref{eq:pit} repeatedly, one easily obtains $\witi{A}$ from $A$ by elementary row operations defined over $R$. 
To prove~\eqref{eq:pit}, note that, as $A$ splits, we have $\pi_{T_j} (A_t) = \sum_{u\in T} \al_u A_u$ for some $\al_u\in R$. Since $A$ is lower-triangular with non-zero diagonal entries, it follows that $\al_u=0$ for all $u\ge t$, so~\eqref{eq:pit} holds.
\end{proof}

Let $\witi{M}$ be the block-diagonal ``truncation'' of $M$ 
defined as in the statement of Lemma~\ref{lem:tr}
with respect to the set partition $\Par(w) = \bigsqcup_{\nu\in\Par'(w)} \Par(w,\nu)$. 
It is well known (and easy to see from the definition) that $M$ is lower-triangular with
respect to the lexicographic order on $\Par(w)$ and that the diagonal entries of $M$ are non-zero. 
Hence, applying Proposition~\ref{prop:Msplit} and Lemma~\ref{lem:tr}, we obtain the following result.

\begin{lem}\label{lem:Mti}
 The matrices $M$ and $\witi{M}$ are row equivalent over $\mZ_{(p)}$. 
\end{lem}

Define $\Pow$ to be the set of all partitions $\lda =(\lda_1,\ldots,\lda_d)$
such that all parts $\lda_i$ are integer powers of $p$. (This includes $1=p^0$.)
 Write $\Pow(w) = \Pow \cap \Par(w)$. 
Define $N=N^{(w)}$ to be the $\Pow(w)\times \Pow(w)$-submatrix of $M^{(w)}$ 

Let 
$\bar{M}=\bar{M}^{(w)}$ be the $\Par(w)\times \Pow(w)$-submatrix of $M^{(w)}$. 
The following result is an immediate consequence of Lemma~\ref{lem:Mti}, due to
the block-diagonal structure of $\witi{M}$ (note that $\Pow(w) = \Par(w, (1^w) )$). 

\begin{lem}\label{lem:Mol}
 The row spaces of $\bar{M}^{(w)}$ and $N^{(w)}$ over $\mZ_{(p)}$ are the same. 
\end{lem}


Define a map $\iota\colon \Par \to \prod_{j\in \mN_{p'}} \Pow$,   
$\;\lda \mapsto (\lda^j)_{j\in \mN_{p'}}$,  by the identity 
$m_{p^n} (\lda^j) = m_{jp^n} (\lda)$
 for all $j\in \mN_{p'}$, $n\ge 0$.
Let $w\ge 0$ and $\nu\in \Par'(w)$. Then $\iota$ restricts to a bijection
from $\Par (w,\nu)$ onto $\prod_{j\in \mN_{p'}} \Pow(m_j (\nu))$, also denoted by $\iota$. 
Let 
\[
 L(\nu) = \bigotimes_{j\in \mN_{p'}} N^{(m_j (\nu))},  
\]
so that $L(\nu)$ is a square matrix with rows and columns indexed by 
$\prod_{j} \Pow(m_j (\nu))$. 
Define a $\Par(w)\times \Par(w)$-matrix $L$ 
by
\begin{equation}\label{eq:defL}
 L_{\lda\mu} = 
\begin{cases}
  L(\nu)_{\iota(\lda),\, \iota(\mu)} & 
\text{if } \lda,\mu\in \Par(w,\nu) \text{ for some } \nu\in \Par'(w), \\
0 & \text{otherwise,}
\end{cases}
\end{equation}
so that $L$ is block-diagonal with respect to the set partition~\eqref{eq:parpart}.

\begin{lem}\label{lem:Ntensor}
 The matrices $M$ and $L$ are  row equivalent over $\mZ_{(p)}$. 
\end{lem}

\begin{proof}
Let $\g\in \Par(w)$ and consider the class function $\xi_{\g}\in \CF(S_w)$ corresponding to row $\g$ of $L$ (via the identification~\eqref{eq:ident}). 
 For each $\eta\in\Par'(w)$ let $P_{\eta}$ be a Sylow $p$-subgroup of $C_{S_w} (g_{\eta})$.
We will verify that $\xi_{\g}$ satisfies the hypothesis of Lemma~\ref{lem:Br1}\eqref{it:2},
 i.e.\ that
\begin{equation}\label{eq:Br1_for_xig}
 \Res^{S_w, g_{\eta}}_{P_{\eta}} \xi_{\g} \in \ol{\cl C}_{(p)} (P_{\eta}) \quad \text{ for all } \eta.
\end{equation}
First, it follows from~\eqref{eq:defL} that $\Res^{S_w, g_{\eta}}_{P_{\eta}} \xi_{\g} =0$ if $\g\notin \Par'(w,\eta)$.

Let $\nu\in \Par'(w)$ be such that $\g\in \Par(w,\nu)$. 
Consider an arbitrary partition $\mu\in \Par(w,\nu)$. Write 
$\iota(\g) = (\g^j)_{j\in \mN_{p'}}$ and $\iota(\mu)=(\mu^j)_{j\in \mN_{p'}}$. 
We have 
\begin{equation}\label{eq:cent}
C_{S_w} ({g_{\nu}}) = \prod_{j\in \mN_{p'}} (C_{j} \wr S_{m_j (\nu)}).
\end{equation}
For each $j$, 
let $G_j$ denote the usual complement to the base subgroup in the wreath product $C_j \wr S_{m_j (\nu)}$, so that $G\simeq S_{m_j (\nu)}$; we view $G_j$
as a subgroup of $S_w$. 
Let $P_j$ be a Sylow $p$-subgroup of $G_j$. Since the index of $\prod_j G_j$ in $C_{S_w} (g_{\nu})$ is prime to $p$, we may assume that 
$P_{\nu}=\prod_j P_j$.
Consider an element $x\in P$ defined by $x= \prod_j x_j$ where $x_j\in P_j$, viewed as an element of $S_{m_j (\nu)}$, has cycle type $\mu^j$.
Let $g_{\nu} = \prod_j g_{\nu, j}$ be the decomposition of $g_{\nu}$ with respect to the direct product~\eqref{eq:cent}, so that $g_{\nu, j}$ has cycle type $(j^{\star m_j (\nu)})$ as an element of $S_w$.
Since $p\nmid j$ and $x_j$ is a $p$-element, it is easy to see that the cycle type of $g_{\nu, j} x_j\in C_j \wr S_{m_j (w)}$ as an element of $S_w$ is obtained
from $\mu^j$ by multiplying each part by $j$. It follows that $\prod_j (g_{\nu,j} x_j)\in S_w$ has cycle type $\mu$. 
Therefore, by~\eqref{eq:defL}, 
\[
\xi_{\g} (g_{\mu}) = \prod_{j\in \mN_{p'}} h_{\g^j} (x_j),
\] 
where $h_{\g^j}$ is viewed as a character of $G_j$ thanks to the identification of that group with $S_{m_j (\nu)}$.
It follows that 
\[
 \Res^{S_w, g_{\nu}}_{P_{\nu}} \xi_{\g} = \bigotimes_{j\in \mN_{p'}} \Res^{G_j}_{P_j} h_{\g^j} \in \cl C (P_{\nu}).
\]
Thus,~\eqref{eq:Br1_for_xig} holds in all cases. 
By Lemma~\ref{lem:Br1}\eqref{it:2}, we have $\xi_{\g} \in \ol{\cl C}_{(p)} (S_w)$.
That is, 
row $\g$ of $L$ belongs to the row space of $M$ over $\ol{\mZ}_{(p)}$. 
Since both $L$ and $M$ are rational-valued, the same holds over $\mZ_{(p)}$.
So the row space of $L$ over $\mZ_{(p)}$ is contained in that of $M$.

However, it is easy to see that $\det(M)=\det(L)$. Indeed, one obtains explicit expressions for $\det(M)$ and $\det(L)$ using the definition of $L$ and the 
fact that the matrices $M$ and $N$ are lower-triangular.
The lemma follows.
\end{proof}

Set $b=b^{(w)}$ to be the $\Pow(w)\times \Pow(w)$-submatrix of $a=a^{(w)}$, so that 
$b^{(w)}_{\lda} = p^{rl(\lda)}$ for all $\lda\in \Pow(w)$. 
Define $Y=Y^{(w)} = NbN^{-1}$ (where $N=N^{(w)}$). 
In Section~\ref{sec:fin} we will prove the following result.

\begin{thm}\label{thm:pow}
 The elementary divisors of $Y$ are $p^{c_{p,r} (\lda)}$, $\lda\in \Pow(w)$.
\end{thm}

Assuming this, we can deduce Theorem~\ref{thm:1t1p} as follows. 
By Lemma~\ref{lem:Ntensor}, $X' = MaM^{-1}$ is equivalent to
$X''=LaL^{-1}$ over $\mZ_{(p)}$. Recall that $L$ is block-diagonal with respect to the set partition
$\Par(w) =\bigsqcup_{\nu\in\Par'(w)} \Par(w,\nu)$. Since $a$ is diagonal, the 
matrix $LaL^{-1}$ is block-diagonal with respect to the same set partition.

Consider the block $X''(\nu)$ of $LaL^{-1}$ corresponding to $\nu\in \Par'(w)$. 
If $\lda\in \Par(w,\nu)$ and $\iota(\lda) = (\lda^j)_{j\in \mN_{p'}}$, then 
$l(\lda) = \sum_j l(\lda^j)$, so 
$a_{\lda} = \prod_j b^{(m_j (\nu) )}_{\lda^j}$. That is, after we apply the identification $\iota$ 
to convert $\Par(w,\nu)\times \Par(w,\nu)$-matrices into 
$(\prod_j \Pow(m_j (\nu)))\times (\prod_j \Pow(m_j (\nu)))$-matrices, the $\nu$-block of $a$ becomes 
equal to $\otimes_j \, b^{(m_j (\nu))}$; and, by~\eqref{eq:defL}, the $\nu$-block of $L$ becomes 
$L(\nu) = \otimes_j \, N^{(m_j (\nu))}$. So $X''(\nu)$ becomes 
$\otimes_{j} Y^{(m_j (\nu))}$. Therefore, by Theorem~\ref{thm:pow}, $X''(\nu)$ is equivalent over $\mZ_{(p)}$ to the diagonal matrix with entries 
$\prod_{j} p^{c_{p,r} (\lda^j)}$, where $(\lda^j)_j$ runs through 
$\prod_{j\in \mN_{p'}} \Pow(m_j (\nu))$. 
But if $\lda\in \Par(w,\nu)$ is such that $\iota(\lda) = (\lda^j)$, then
$m_{jp^t} (\lda) = m_{p^t} (\lda^j)$ for all $j\in \mN_{p'}$ and $t\ge 0$, and so
\[
 \begin{split}
  c_{p,r} (\lda) &= \sum_{\substack{n\in \mN \\ 0\le v_p (n)<r }} 
\Big( (r-v_p (n)) m_n(\lda) + d_p(m_n (\lda)) ) \Big) \\
&= \sum_{j\in \mN_{p'}} \sum_{t=0}^{r-1} \Big( (r-t) m_{p^t} (\lda^j) + d_p (m_{p^t} (\lda^j)) \Big) \\
& = \sum_{j\in \mN_{p'}} c_{p,r} (\lda^j).
 \end{split}
\]
(The second equality is obtained by substituting $n=jp^t$.) 
Hence, $X''(\nu)$ is equivalent  to $\diag\{ (p^{c_{p,r} (\lda)})_{\lda\in\Par(w,\nu)} \}$
over $\mZ_{(p)}$. Therefore, $X''$ is equivalent to 
$\diag\{ (p^{c_{p,r} (\lda)})_{\lda\in \Par(w)} \}$ over $\mZ_{(p)}$, and hence over $\mZ$ 
(see Remark~\ref{rem:Zp}). Since $X$ is equivalent to $X''$, we have shown that 
Theorem~\ref{thm:1t1p} is implied by Theorem~\ref{thm:pow}.

\section{Proof of Theorem~\ref{thm:pow}}\label{sec:fin}

Recall that Theorem~\ref{thm:pow} is concerned with the $\Pow(w)\times \Pow(w)$-matrix $Y=NbN^{-1}$ where 
$N_{\lda\mu} = |\scr M_{\lda\mu}|$ (cf.\ the definition before Eq.~\eqref{eq:M}) and $b_{\lda} = p^{r l(\lda)}$ for all 
$\lda,\mu\in \Pow(w)$. 
Let $z= \diag\{ (z_{\lda})_{\lda \in \Pow(w)} \}$ (see~\eqref{eq:defz}).

\begin{lem}\label{lem:Nequiv}
 The matrix $N$ is row equivalent to $(N^{\tr})^{-1}z$ over $\mZ_{(p)}$.
\end{lem}

\begin{proof}
 Let $\mtt m = (m_{\lda})$ be the graded basis of $\CF(S_w)$ 
 such that $(\mtt h, \mtt m)$ is a
dual pair (cf.~\cite[Chapter I, Eq.\ (4.5)]{Macdonald1995}). Since $\mtt h$ is a $\mZ$-basis of $\cl C(S_w)$, the same is true for $\mtt m$. 
Hence, the transition matrix $M(\mtt m, \ti{\mtt p}; w)$ is row equivalent to $M$ (recall that
$M = M(\mtt h,\ti{\mtt p}; w)$). Since $(\mtt h,\mtt m)$ and $(\mtt p,\ti{\mtt p})$ are dual pairs,
$M(\mtt m, \mtt p; w)=(M^{\tr})^{-1}$. Hence, $M(\mtt m, \ti{\mtt p}; w) = (M^{\tr})^{-1} \hat{z}$, 
where $\hat{z} = \diag\{ (z_{\lda})_{ \lda\in \Par(w) } \}$. 

So $(M^{\tr})^{-1} \hat{z}$ is row equivalent to $M$ (over $\mZ$). On the other hand,
by Lemma~\ref{lem:Mti}, there exists $U\in \GL_{\Par(w)} (\mZ_{(p)})$ such that $M=U\witi{M}$.
We have
\[
(M^{\tr})^{-1} \hat{z} = ((U\witi{M})^{\tr})^{-1} \hat{z} = 
(U^{\tr})^{-1} \Big( (\witi{M}^{\tr})^{-1} \hat{z} \Big). 
\]
Therefore, $(\witi{M}^{\tr})^{-1} \hat{z}$ is row equivalent over $\mZ_{(p)}$ to $M$, and hence to $\witi{M}$.
But $\witi{M}$ and $(\witi{M}^{\tr})^{-1} \hat{z}$ are both block-diagonal with respect to the decomposition
$\Par(w) = \bigsqcup_{\nu\in \Par'(w)} \Par(w,\nu)$; and the blocks of these two matrices corresponding to
$\nu=(1^w)$ are $N$ and $(N^{\tr})^{-1} z$ respectively. The result follows. 
\end{proof}

Due to Lemma~\ref{lem:Nequiv}, $Y=NbN^{-1}$ is equivalent over $\mZ_{(p)}$ to 
\begin{equation}\label{eq:Ydash}
 Y' = Nb ((N^{\tr})^{-1} z)^{-1} = N b z^{-1} N^{\tr}.
\end{equation}

Let $\lda\in \Pow$. In the sequel, we will write $n_{i} (\lda) = m_{p^i} (\lda)$ for all $i\ge 0$.
 We define partitions 
$\lda^{<r}, \lda^{\ge r}, \bar{\lda} \in \Pow$ as follows: for all $i\ge 0$, 
 \begin{align}
     n_i (\lda^{<r}) &= 
      \begin{cases}
	  n_i (\lda) & \text{if } i<r, \\
	  0 & \text{otherwise,}
      \end{cases} \\
      n_i (\lda^{\ge r}) &= n_{r+i} (\lda), \\
      n_i (\bar{\lda}) &= 
	\begin{cases}
	  n_i (\lda) & \text{if } i<r, \\
	  \sum_{j \ge r} p^{j-r} n_j(\lda) & \text{if } i=r, \\
	  0 & \text{if } i\ge r. 
	\end{cases}
 \end{align}
(Thus, $|\bar\lda|=|\lda|$ and $\bar{\lda}$ is obtained from $\lda$ by splitting all parts of size at least $p^r$ into parts of size exactly $p^r$.) Note that
$|\lda| = |\lda^{<r}| + p^r |\lda^{\ge r}|$. 

Let $\cl K$ denote the set of all $\ka \in \Pow(w)$ such that 
$\ka = \bar{\ka}$ (i.e.\ $n_i (\ka)=0$ for all $i>r$). 
For each $\ka\in \cl K$ define
\[
 \Pow_{\ka} = \{ \lda\in \Pow(w) \mid \bar{\lda} = \ka \}. 
\]
 We have $\Pow(w) = \bigsqcup_{\ka\in\cl K} \Pow_{\ka} (w)$.
 In the sequel, ``blocks'' of a $\Pow(w)\times \Pow(w)$-matrix 
 are understood to be ones corresponding to this partition of $\Pow(w)$.
 In particular, a $\Pow(w)\times \Pow(w)$-matrix
$Z$ is said to be \emph{block-diagonal} if $Z_{\lda\mu} =0$ whenever 
$\bar\lda \ne \bar\mu$. Further, $Z$ is \emph{block-scalar} if 
$Z_{\lda\mu} = \al_{\bar{\lda}} \de_{\lda\mu}$ for all $\lda,\mu\in \Pow(w)$,
where $(\al_{\ka})_{\ka\in \cl K}$ is a tuple of rational numbers.

\begin{remark}\label{rem:simplecase}
In the case when $p^r >w$, we have 
$\Pow_{\ka} (w) = \{ \ka \}$ for all $\ka\in \cl K$, and the proof below becomes much simpler
(in particular, see Remark~\ref{rem:Usimpl}).
The reader may find it helpful to consider the case $p^r>w$ in the first 
instance. Roughly speaking, the proof in the general case is obtained
by applying the (trivial) proof for the case $r=0$ ``within blocks'' and 
the proof for the case $p^r>w$ ``between blocks''.
\end{remark}

For each $\lda\in \Pow$ define
\begin{align*}
  x_{\lda} &=  \prod_{i\ge 0} n_i(\lda)! \qquad\text{and} \\
  y_{\lda} &= \prod_{i\ge 0} p^{i n_i (\lda)}.
 \end{align*}
Note that $z_{\lda} = x_{\lda} y_{\lda}$.
Define $x= \diag\{ (x_{\lda})_{\lda\in \Pow(w)} \}$
and $y = \diag\{ (y_{\lda})_{\lda\in\Pow(w)} \}$.

Define diagonal $\Pow(w)\times \Pow(w)$-matrices $x^{<r}, x^{\ge r}, y^{<r}, y^{\ge r}, \ti{y}$ as follows:
for all $\lda\in \Pow(w)$, 
\begin{align*}
   x^{<r}_{\lda} &= \prod_{0\le i<r} n_i(\lda)! = x_{\lda^{<r}}, \\
  x^{\ge r}_{\lda} &= \prod_{i\ge r} n_i(\lda)! = x_{\lda^{\ge r}}, \\
  y^{<r}_{\lda} &= \prod_{0\le i<r} p^{in_i(\lda)} = y_{\lda^{<r}}, \\
  y^{\ge r}_{\lda} &= \prod_{i\ge r} p^{(i-r)n_i(\lda)} = y_{\lda^{\ge r}}, \\
  \ti{y}_{\lda} &= \prod_{i\ge r} p^{rn_i(\lda)}.
\end{align*}
It is easy to verify that 
\begin{align}
  x &= x^{<r} x^{\ge r} \qquad \quad \text{and} \label{eq:x} \\
  y &= y^{<r}  y^{\ge r}  \ti{y}.\label{eq:y}
\end{align}
Define a $\Pow(w)\times \Pow(w)$-matrix $C$ as follows: 
\begin{equation}\label{eq:defC}
 C_{\lda\mu} = 
\begin{cases}
 N_{\lda^{\ge r}, \, \mu^{\ge r}}^{( |\lda^{\ge r}| )} & \text{if } \bar{\lda} = \bar{\mu}, \\
 0 & \text{otherwise,}
\end{cases}
\end{equation}
so that $C$ is block-diagonal. 
For each $\ka\in\cl K$ let $C(\ka)$ be the $\Pow_{\ka} (w)\times \Pow_{\ka}(w)$-submatrix 
of $C$. Let $A = NC^{-1}$, so that 
\begin{equation}\label{eq:AC}
 N = AC.
\end{equation}  
Let 
\begin{equation}\label{eq:blessr}
b^{<r} = b\tilde{y}^{-1},
\end{equation}
so that $b^{<r}_{\lda} = p^{rl(\lda^{<r})}$ for all $\lda$.
Note that $b^{<r}$, $x^{<r}$ and $y^{<r}$ are block-scalar, and hence these matrices commute with $C$.

Let $\ka\in \cl K$. We have a bijection from $\Pow_{\ka} (w)$ onto $\Pow(n_r(\ka))$ given by 
$\lda \mapsto \lda^{\ge r}$. After relabelling of rows and columns via this bijection, $C(\ka)$ becomes 
$N^{(n_r (\ka))}$. Hence, by Lemma~\ref{lem:Nequiv}, $C(\ka)$ is row equivalent over $\mZ_{(p)}$ to 
the matrix $((C(\ka))^{\tr})^{-1} x^{\ge r}(\ka) y^{\ge r}(\ka)$, where
$x^{\ge r}(\ka)$ and $y^{\ge r}(\ka)$ are the $\Pow_{\ka} (w)\times \Pow_{\ka}(w)$-submatrices of $x^{\ge r}$ and
$y^{\ge r}$ respectively. So there is 
$S{(\ka)} \in \GL_{\Par_{\ka} (w)} (\mZ_{(p)})$ such that
\[
   ((C(\ka))^{\tr})^{-1} x(\ka) y(\ka) = S{(\ka)} C{(\ka)}.
\]
Let $S$ be the block-diagonal $\Pow(w)\times \Pow(w)$-matrix with the $\ka$-block equal to 
$S{(\ka)}$ for each $\ka$. Then
\begin{equation}\label{eq:Ctr}
  (C^{\tr})^{-1} x^{\ge r} y^{\ge r} = SC.
\end{equation}
Define
\begin{equation}\label{eq:B}
 B = S^{-1} A^{\tr} S.
\end{equation}
We have
\begin{flalign}
 Y' &=  \mathrlap{Nbx^{-1} y^{-1}N^{\tr}} &  \text{by~\eqref{eq:Ydash}} \notag \\
&= \mathrlap{AC b\ti{y}^{-1} (x^{<r})^{-1} (y^{<r})^{-1} (x^{\ge r})^{-1} (y^{\ge r})^{-1} C^{\tr} A^{\tr} }
& \text{by~\eqref{eq:x}, \eqref{eq:y}, \eqref{eq:AC}}  \notag \\
 &= \mathrlap{ AC b^{<r} (x^{<r})^{-1} (y^{<r})^{-1} ((C^{\tr})^{-1} x^{\ge r} y^{\ge r})^{-1}  A^{\tr} }
& \text{by \eqref{eq:blessr}} \notag \\
 &= \mathrlap{ AC b^{<r}  (x^{<r})^{-1} (y^{<r})^{-1} C^{-1} S^{-1} A^{\tr} }
& \text{by \eqref{eq:Ctr}} \notag \\
&= \mathrlap{ A b^{<r}  (x^{<r})^{-1} (y^{<r})^{-1} S^{-1} A^{\tr} }
& \text{since } C \text{ commutes with } b^{<r},\, x^{<r},\, y^{<r}  \notag \\
 \quad &= \mathrlap{ A b^{<r} (x^{<r})^{-1} (y^{<r})^{-1} B S^{-1} }
& \text{by \eqref{eq:B}}. \label{eq:long}
\end{flalign} 
Let $U= (x^{<r})^{-1} A$, so that
\begin{equation}\label{eq:U}
  A = x^{<r} U.
\end{equation}
 Then
\[
 B = S^{-1} A^{\tr} S = S^{-1} U^{\tr} x^{<r} S = S^{-1} U^{\tr} S x^{<r}
\]
because $S$ commutes with $x^{<r}$ (as $S$ is block-diagonal and $x^{<r}$ is block-scalar). 
Therefore, defining 
\begin{equation}\label{eq:V}
 V = S^{-1} U^{\tr} S, 
\end{equation}
we have $B=V x^{<r}$. Substituting this and~\eqref{eq:U} into~\eqref{eq:long}, we obtain
\[
 Y' = x^{<r} U b^{<r} (x^{<r})^{-1} (y^{<r})^{-1} V x^{<r} S.
\]
Since $S\in \GL_{\Pow(w)} (\mZ_{(p)})$, the matrix
\begin{equation}\label{eq:Ydd}
 Y'' = x^{<r} U b^{<r} (x^{<r})^{-1} (y^{<r})^{-1} V x^{<r}
\end{equation}
is equivalent to $Y'$, and hence to $Y$, over $\mZ_{(p)}$.

\begin{remark}
 If we remove $U$ and $V$ from the product on the right-hand side of~\eqref{eq:Ydd} and simplify the resulting expression, we are left with
$b^{<r} x^{<r} (y^{<r})^{-1}$. 
An easy calculation shows that $v_p (b^{<r}_{\lda} x^{<r}_{\lda} (y^{<r}_{\lda})^{-1}) = c_{p,r} (\lda)$
for all $\lda \in \Pow(w)$ (see~\eqref{eq:rho} below).
Hence, to prove Theorem~\ref{thm:pow}, it is enough to show that  removing $U$ and $V$ from the 
product~\eqref{eq:Ydd} does not change the invariant factors.
Lemma~\ref{lem:invfac} gives general sufficient conditions for this to be true for products of this kind.
The fact that these conditions hold in our case is established at the end of the 
paper using 
Lemma~\ref{lem:U}, which gives detailed information on the entries of $U$.
\end{remark}

For each $\lda\in \Pow$, define
 \begin{align}
 e_{\lda} &= \sum_{i=0}^{r-1} d_p (n_i (\lda)),   \notag \\
 f_{\lda} &= \sum_{i=0}^{r-1} (r-i) n_i (\lda), \qquad \text{and}  \label{eq:deff} \\  
 k_{\lda} & = f_{\lda} - e_{\lda}. \notag
\end{align}
Note that  
\begin{align}
 e_{\lda} &= v_p(x^{<r}_{\lda}),  \label{eq:elambda} \\
 f_{\lda} & = v_p (b_{\lda}^{<r}) - v_p(y^{<r}_{\lda}), \qquad \qquad \text{and} \label{eq:flambda} \\
 \qquad c_{p,r} (\lda) & = \sum_{0\le i<r} \Big( (r-i)n_i(\lda) + d_p (n_i(\lda)) \Big) = f_{\lda} +e_{\lda}. 
\label{eq:elplfl}
\end{align}


\begin{lem}\label{lem:U}
For all $\lda,\mu \in \Pow(w)$:
\begin{enumerate}[(i)]
\item\label{T1} $U_{\lda\mu} \in \mZ_{(p)}$;
 \item\label{T3} $U_{\lda\mu} = \de_{\lda\mu}$ if $\bar{\lda} = \bar{\mu}$;
 \item\label{T4} $v_p ( U_{\lda\mu} ) > k_{\lda} -k_{\mu}$ if $\bar{\lda} \ne \bar{\mu}$. 
\end{enumerate}
\end{lem}

\begin{proof}
We begin with~\eqref{T3}. Consider the block $\Pow_{\ka}$ for a fixed $\ka \in \cl K$. 
Write $N(\ka)$, $A(\ka)$, $U(\ka)$ for the $\Pow_{\ka}\times \Pow_{\ka}$-submatrices of $N,A,U$ respectively. 
We have
\begin{equation}\label{eq:NxC}
N(\ka) = x_{\ka} C(\ka),
\end{equation}
that is to say, $N_{\lda\mu} = x_{\ka} N_{\lda^{\ge r}\mu^{\ge r}}^{(|\lda^{\ge r}|)}$ for all $\lda,\mu\in \Pow_{\ka}$.
Indeed, every $f\in \scr M_{\lda\mu}$ satisfies $\lda_{f(i)} = \mu_{i}$ for all $i\in [1,l(\mu)]$ such that $\mu_i<p^r$. 
Hence, such a map $f$ is determined by an element of $\scr M_{\lda^{\ge r}, \mu^{\ge r}}$ together with 
a permutation of the set $\{ i \mid \mu_i = p^t \}$ for each $t\in [0,r-1]$ (and the correspondence is bijective). 

Using~\eqref{eq:NxC} and the definition of $U$ (cf.~\eqref{eq:AC} and~\eqref{eq:U}), we obtain
\[
 U(\ka) = x_{\ka}^{-1} N(\ka)C(\ka)^{-1} = \mathbb I_{\Pow_{\ka}},
\]
so~\eqref{T3} holds.

In order to prove~\eqref{T1} and~\eqref{T4}, we first need to 
establish a decomposition of $N$ as a product, which may be informally described as follows.
If $\lda,\mu\in \Pow (w)$, then an element of $\scr M_{\lda\mu}$ may be viewed as a 
way to aggregate the parts $\mu_j$ into ``lumps'' and to associate bijectively some $i\in [1, l(\lda)]$ 
with each lump 
in such a way that $\lda_i$ is the sum of the parts $\mu_j$ in the lump. This process may be split into
two stages: first, aggregate the parts $\mu_j \ge p^r$ that are supposed to go to the same lump, without
touching the parts $\mu_j <p^r$; then, aggregate the parts $\mu_j <p^r$ with each other and with 
the lumps obtained in the first stage to obtain the desired element of $\scr M_{\lda\mu}$. 
This leads to a decomposition of $N$ as a product of two matrices. The following construction makes
this argument precise. 

Define $\Pow_{<r}$ to be the set of $\lda\in \Pow$ such that $\lda_1<p^r$ 
(or, equivalently, $\lda^{<r}=\lda$). 
Let 
\[
 \begin{split}
  \cl P &= \{ (\eta, \theta) \in \Pow_{<r} \times \Pow \mid 
|\eta| + p^r |\theta| =  w \} \quad \text{and} \\
  \cl Q &= \{ (\eta, \theta) \in \Pow_{<r} \times \Par \mid 
|\eta| + p^r |\theta| = w \}.
 \end{split}
\]
Note that  $\lda \mapsto (\lda^{<r}, \lda^{\ge r})$ is a bijection  from $\Pow(w)$ onto $\cl P$. 
For every $\Pow(w)\times \Pow(w)$-matrix $Z$, we write $Z^{\star}$ for the 
$\Pow(w)\times \cl P$-matrix obtained from $Z$ by relabelling the set of columns via this bijection; 
and $Z^{\star\star}$ denotes the $\cl P\times \cl P$-matrix obtained by relabelling both rows and columns.
Let $\iota\colon \cl P \to \Pow(w)$ be the inverse of our bijection.  
 
Let $D$ be the $\cl Q\times \cl P$-matrix defined by 
\[
D_{(\eta^1, \theta^1),\, (\eta^2, \theta^2)} = \de_{\eta^1\eta^2} M_{\theta^1 \theta^2}^{(|\theta_1|)}.
\]
For every $\lda\in \Pow(w)$ and $(\eta,\theta) \in \cl Q$ let 
$\scr E_{\lda, (\eta,\theta)}$ be the set of all pairs $(f,g)$ of maps $f\colon [1,l(\eta)]\to [1,l(\lda)]$ and
$g\colon [1,l(\lda)] \to \mZ_{\ge 0}$ such that 
\begin{enumerate}[(a)]
 \item\label{it:E1} $p^r g(t) +\sum_{i\in f^{-1} (t) } \eta_i = \lda_t$ for all $t\in [1,l(\lda)]$;
 \item\label{it:E2} $m_{u} (\theta) = |g^{-1} (u)|$ for all $u\in \mN$. 
\end{enumerate}
(We remark that for every $f$ there is at most one $g$ such that $(f,g) \in \scr E_{\lda, (\eta,\theta)}$,
due to~\eqref{it:E1}.)
Set $E_{\lda, (\eta,\theta)} = |\scr E_{\lda,(\eta,\theta)}|$, so that $E$ is a $\Pow(w)\times \cl Q$-matrix. 
We claim that 
\begin{equation}\label{eq:N=AD}
  N^{\star} = ED.
\end{equation}
Indeed, for every $\lda, \mu \in \Pow(w)$, a bijection 
\begin{equation}\label{eq:bij}
 \scr M_{\lda\mu} \longleftrightarrow \bigsqcup_{\theta\in \Par(|\mu^{\ge r}|)} 
\scr E_{\lda, (\mu^{<r},\theta)} \times \scr M_{\theta\mu^{\ge r}} 
\end{equation}
is constructed by the following rule.
Let $((f,g),h)$ belong to the right-hand side of~\eqref{eq:bij}, so that
for some $\theta\in \Par(|\mu_{\ge r}|)$ we have $(f,g) \in \scr E_{\lda, (\mu^{<r},\theta)}$ and $h\in \scr M_{\theta\mu^{\ge r}}$. 
If $q\in \scr M_{\lda \mu}$, 
we set $q \leftrightarrow ((f,g),h)$ if the following five conditions are satisfied:
\begin{enumerate}[(1)]
 \item\label{it:sim1} $f(i) = q(i + l(\mu^{\ge r}))$ for all $i\in [1, l(\mu^{<r})]$;
 \item\label{it:sim2} $g(t) =  \displaystyle\sum\limits_{ \substack{j\in [1,l(\mu^{\ge r})] \\ q(j) = t} } \mu_j$ for all $t\in [1,l(\lda)]$;
 \item\label{it:sim25} $\theta_{h(j)} = g(q(j))$ for all $j\in [1,l(\mu^{\ge r})]$.
 \item\label{it:sim3} $h(j) = h(j')$ if and only if $q(j) = q(j')$ for $j,j'\in [1,l(\mu^{\ge r})]$;
 \item\label{it:sim4} if $j,j'\in [1,l(\mu^{\ge r})]$ and $\theta_{h(j)}=\theta_{h(j')}$, then $h(j)<h(j')$ if and only if $q(j)<q(j')$.
\end{enumerate}
 (With regard to~\eqref{it:sim1} and~\eqref{it:sim2}, note that, for $i\in [1,l(\mu)]$, 
one has $\mu_i <p^r$ if and only if $i> l(\mu^{\ge r})$. Condition~\eqref{it:sim25} follows from the other ones and is included for clarity.) It is routine to verify that this rule does yield a 
bijection~\eqref{eq:bij}, proving~\eqref{eq:N=AD}.

We are now able to prove~\eqref{T1}.
Let $\eta \in \Pow_{<r}$ be such that $w-|\eta| = p^r u$ for some $u\in \mZ_{\ge 0}$.
Let $\cl P(\eta)$ (respectively, $\cl Q(\eta)$) be the set of elements of $\cl P$ (respectively, $\cl Q$)
with first coordinate $\eta$. 
Then $\cl P(\eta)$ and $\cl Q(\eta)$ may be identified with $\Pow(u)$ and $\Par(u)$ respectively via 
projection onto the second coordinate. 
Under this identification, the $\cl Q(\eta)\times \cl P(\eta)$-submatrix of $D$ 
becomes equal to $\bar M^{(u)}$.
By Lemma~\ref{lem:Mol}, $\bar M^{(u)}$ has the same row space over $\mZ_{(p)}$ as $N^{(u)}$.
Hence, $D = D'C^{\star\star}$ for some $\cl Q\times \cl P$-matrix $D'$ with entries 
in $\mZ_{(p)}$  (cf.~\eqref{eq:defC}). Due to~\eqref{eq:N=AD}, we obtain
$N^{\star} = ED' C^{\star\star}$. Using~\eqref{eq:AC}, we deduce that
$A^{\star} = ED'$, and hence, by~\eqref{eq:U},
\begin{equation}\label{eq:Ustar}
 U^{\star} = (x^{<r})^{-1} ED'.
\end{equation}

Let $\lda\in \Pow(w)$. For each $j\in [0,r-1]$, consider the  group $S_{n_j (\lda)}$
of all permutations of the set $\{ i\mid \lda_i = p^j \}$. 
For any $(\eta,\theta)\in \cl Q$, the group 
 $\prod_{0\le j<r} S_{n_j (\lda)}$ acts on 
$\scr E_{\lda, (\eta, \theta)}$ by $\s \cdot (f,g) = (\s \circ f, g)$.
This action is free because whenever $0\le j<r$ and $\lda_i = p^j$ one has $f^{-1} (j) \ne \varnothing$
for any $(f,g) \in \scr E_{\lda, (\eta,\theta)}$ (due to condition~\eqref{it:E1}).
The order of the group is $x_{\lda}^{<r}$, and therefore $x_{\lda}^{<r}$ divides $E_{\lda, (\eta,\theta)}$.
Due to~\eqref{eq:Ustar}, this implies that the entries of $U^{\star}$ lie in $\mZ_{(p)}$, 
proving~\eqref{T1}.

Finally, we prove~\eqref{T4}. 
Suppose that $\lda,\mu\in \Pow(w)$ and $\bar{\lda} \ne \bar{\mu}$, i.e.\ $\lda^{<r}\ne \mu^{<r}$. 
Fix $\theta\in \Pow(|\mu^{\ge r}|)$. 
We partition the set $\scr E_{\lda, (\mu^{<r}, \theta)}$ as follows. 
Let $\cl G$ be the set of all pairs $(g,\g)$ of maps
$g\colon [1,l(\lda)] \to \mZ_{\ge 0}$ and $\g\colon [1,l(\lda)]\to \Pow_{<r}$ such that
\begin{enumerate}[(A)]
 \item $m_{u} (\theta) = |g^{-1} (u)|$ for all $u\in \mN$;
 \item\label{it:G2} $p^r g(t) + |\g(t)| = \lda_t$ for all $t\in [1,l(\lda)]$;
\item\label{it:G3} $\mu^{<r} = \sum_{t=1}^{l(\lda)} \g(t)$.
\end{enumerate}
For each such pair $(g,\g)$ let $\scr E^{g,\g}$ be the set of maps 
$f\colon [1,l(\mu^{<r})]\to [1,l(\lda)]$ such that 
for every $t\in [1,l(\lda)]$ the partition $\g(t)$ is obtained by rearranging the multiset
$\{ \mu^{<r}_j \mid j\in f^{-1} (t) \}$ in the non-decreasing order.
It follows from the definitions that 
\begin{align}
  \scr E_{\lda, (\mu^{<r}, \theta) } & = \bigsqcup_{(g,\g)\in \cl G} \{ (f,g) \mid f\in \scr E^{g,\g} \},
\notag \\
\text{so}  \qquad 
E_{\lda, (\mu^{<r},\theta)} &= \sum_{(g,\g)\in \cl G} |\scr E^{g,\g}|. \label{eq:Esqcup}
\end{align}
Moreover, it is clear from the definition of $\scr E^{g,\g}$ that
\[
 |\scr E^{g,\g}| = 
\prod_{j=0}^{r-1} \binom{n_j (\mu)}{n_j (\g(1)), n_j (\g(2)), \ldots, n_j (\g(l(\lda)))},
\]
whence
\begin{equation}\label{eq:Eggamma}
 v_p (|\scr E^{g,\g}|) = e_{\mu} - \sum_{t=1}^{l(\lda)} e_{\g(t)} \quad \text{ for all } (g,\g) \in \cl G. 
\end{equation}

We will estimate each $e_{\g (t)}$ from above using the general inequality 
\begin{equation}\label{eq:dineq}
 d_p (s) < s \quad \text{ for all } s\in \mN,
\end{equation}
which is an easy consequence of~\eqref{eq:dp}.
Consider any $t\in [1,l(\lda)]$ such that $\g(t)\ne \varnothing$, and set $i_0 = i_0 (t) = \min\{ i \mid n_i (\g(t))\ne 0 \}$. 
We have
\[
 f_{\g(t)} - e_{\g(t)} = \sum_{i=0}^{r-1} \Big( (r-i) n_i (\g(t)) - d_p (n_i (\g(t)) \Big) > \sum_{i=0}^{r-1} (r-i-1) n_i (\g(t)) \ge r-i_0-1, 
\]
where the strict inequality is due to~\eqref{eq:dineq}.
Hence,
\begin{equation}\label{eq:i0}
 f_{\g(t)} - e_{\g(t)} \ge r-i_0 (t).
\end{equation}
Since $\g(t)\in \Pow_{<r}$, we have $i_0 (t)<r$.
Moreover, if $\lda_t < p^r$, then by~\eqref{it:G2} we have $|\g(t)|=\lda_t$, whence $i_0 (t) \le \log_p \lda_t$. 
Hence, summing~\eqref{eq:i0} over all $t\in [1,l(\lda)]$ such that $\g(t)\ne \varnothing$, we obtain
\begin{equation}\label{eq:T4_1}
 \sum_{t=1}^{l(\lda)} f_{\g(t)} - \sum_{t=1}^{l(\lda)} e_{\g(t)} \ge \sum_{t=1}^{l(\lda)} \max\{0, r-\log_p \lda_t\} = f_{\lda},
\end{equation}
where the equality is an obvious consequence of~\eqref{eq:deff}. 

We claim that the inequality~\eqref{eq:T4_1} is strict. If not, then by analysing the argument leading up to it, we see that for every $t\in [1,l(\lda)]$ such that $\lda_t<p^r$, we 
must have $i_0 (t) = \lda_t$, whence $\g(t) = (\lda_t)$. (For such $t$, we have $|\g(t)|=\lda_t$, so the condition that $\g(t)\ne \varnothing$ holds automatically.) 
But due to~\eqref{it:G3}, this implies that $\mu^{<r}=\lda^{<r}$, a contradiction. 

Combining the (strict) inequality~\eqref{eq:T4_1} with Eq.~\eqref{eq:Eggamma}, we obtain
\[
 v_p (|\scr E^{g,\g}|) > e_{\mu} +f_{\lda} - \sum_{t=1}^{l(\lda)} f_{\g(t)} =e_{\mu} +f_{\lda} - f_{\mu},
\]
where the equality holds by~\eqref{it:G3} and~\eqref{eq:deff}. 
By~\eqref{eq:Esqcup}, it follows that 
\[
v_p (E_{\lda, (\mu^{<r},\theta)} ) >e_{\mu}+f_{\lda}-f_{\mu}
\qquad \text{(for all } \theta\in \Par(|\mu^{\ge r}|) \text{)}.
\] 
Due to~\eqref{eq:Ustar}, we deduce that 
\[
 v_p (U_{\lda\mu} ) > v_p ((x^{<r}_{\lda})^{-1}) +e_{\mu} +f_{\lda} -f_{\mu} =
-e_{\lda} +e_{\mu} +f_{\lda} - f_{\mu} = k_{\lda} -k_{\mu}.
\qedhere
 \]
\end{proof}

\begin{remark}\label{rem:Usimpl}
 In the special case when $p^r >w$, the proof of Lemma~\ref{lem:U} is much simpler and may be sketched as follows. First, parts~\eqref{T1} and \eqref{T3} 
are obvious. (Note that in the given case $C = \mathbb I_{\Pow(w)}$, and so 
$U = x^{-1} N$: see~\eqref{eq:AC} and~\eqref{eq:U}.)
Secondly, part~\eqref{T4} follows from part~\eqref{T1} together with the facts
that $U_{\lda\mu}=0$ if $\scr M_{\lda\mu}=\varnothing$ and $k_{\lda} < k_{\mu}$ if $\scr M_{\lda\mu} \ne \varnothing$ and $\bar \lda \ne \bar \mu$. 
The latter inequality can easily be proved by reducing to the case when $\mu$ is obtained from $\lda$ by replacing one part $p^j$ (for some $j>0$) with $p$ parts of size $p^{j-1}$.
\end{remark}

\begin{lem}\label{lem:invfac}
Let $R$ be a discrete valuation ring with field of fractions $K$ and valuation 
$v\colon K\to \mZ\cup\{\infty \}$.
Let $I$ be a finite set. Suppose that $s,t,u, P, Q\in \GL_{I} (K)$ and $s,t,u$ are diagonal. 
Set $\rho_i = v(s_i) +v(t_i)+v(u_i)$ for all $i\in I$. 
Suppose that there exist tuples $(\al_i)_{i\in I}$ 
and $(\be_i)_{i\in I}$ of rational numbers such that for all $i,j\in I$ the following hold:
\begin{enumerate}[(i)]
 \item\label{if1} $v (t_i) = \al_i - \be_i$;
 \item\label{if2} $v (P_{ij}- \de_{ij}) > \al_i - \al_j$;
 \item\label{if3} $v(Q_{ij} - \de_{ij}) > \be_i-\be_j$;
 \item\label{if4} if $\rho_i\ge \rho_j$, then $\al_i - \al_j \ge v(s_{j}) - v(s_{i})$;
 \item\label{if5}  if $\rho_{i}\ge \rho_{j}$, then $\be_{j} - \be_i \ge v(u_j) - v(u_i)$.
\end{enumerate}
Then $sPtQu$ is equivalent to $stu$ over $R$. 
\end{lem}

\begin{proof}
Let $\pi$ be a uniformising element of $R$. For $d\in \mN$, consider the simple extension $K'$
of $K$ generated by a $d$-th root of $\pi$, and let $R'$ be the integral closure of $R$ in $K'$.
Then $R'$ is a discrete valuation ring (see e.g.~\cite[Chapter 1, Proposition 17]{Serre1979}). 
If we view all the matrices in the lemma as ones
with entries in $K'$ rather than $K$, then all valuations are multiplied by $d$. Thus, choosing an
appropriate $d$, we may assume that $\al_i$ and $\be_i$ are integers for all $i$. 

Let $Z = sPtQu$. 
By~\eqref{if1}, we can represent $t$ as a product of two diagonal matrices $t^{(1)}$ and $t^{(2)}$ such that $v(t^{(1)}_{i})=\al_{i}$ and $v(t^{(2)}_{i}) = -\be_{i}$ for all $i\in I$. 
Let $P'=(t^{(1)})^{-1} P t^{(1)}$ and $Q'=t^{(2)} Q (t^{(2)})^{-1}$, so that 
$Z=st^{(1)} P'Q' t^{(2)} u$.
Consider the following subgroup $\G$ of $\GL_{I} (R)$:
\[
 \G = \{ J \in \GL_{I} (K) \mid v(J_{ij} - \de_{ij}) > 0 \; \text{ for all } i,j\in I \}. 
\]
We have $P'\in \G$. Indeed, for all $i,j\in I$, 
\[
 v(P'_{ij} - \de_{ij} ) = -v(t^{(1)}_i) + v(P_{ij} -\de_{ij}) + v(t^{(1)}_j) 
= -\al_i +v(P_{ij} - \de_{ij}) + \al_j >0 \qquad \text{by~\eqref{if2}}.
\]
Similarly, $Q'\in \G$ by~\eqref{if3}. So $P'Q'\in \G$.

Fix a total order $\le$ on $I$ such that $i\le j$ implies $\rho_i \le \rho_j$ for all $i,j\in I$. 
Using standard Gaussian elimination, one can decompose any element of $\G$ as a product of a 
lower-triangular and an upper-triangular matrix (with respect to this order) 
such that both matrices belong to $\G$. 
In particular, 
$P'Q' = JH$ for some lower-triangular $J\in \G$ and upper-triangular $H\in \G$.
Let $J' = st^{(1)} J (st^{(1)})^{-1}$ and $H'= (t^{(2)} u)^{-1} H t^{(2)} u$. Then
\[
Z = st^{(1)} JH t^{(2)} u = J' st^{(1)} t^{(2)} u H'  = J'stuH'.
\]
Now $J'$ is lower-triangular and $v(J'_{ii}-1)>0$ for all $i\in I$ because $J$ has the 
same properties. Further, if $i>j$ are elements of $I$, then $\rho_i\ge \rho_j$, and hence
\begin{align*}
 v(J'_{ij}) &= v(s_i)+ v(t^{(1)}_i) + v(J_{ij}) - v(s_j) - v(t^{(1)}_j) && \\
 & = v(J_{ij}) + \al_i - \al_j + v(s_i) - v(s_j) \ge v(J_{ij}) >0 && \text{by~\eqref{if4}}.
\end{align*}
Hence, $J'\in \G \le \GL_I (R)$. By a similar argument, it follows from~\eqref{if5} that
$H'\in \GL_I (R)$. Therefore, $Z$ is equivalent to $stu$ over $R$.
\end{proof}

We are now in a position to complete the proof of Theorem~\ref{thm:pow}.
We will apply Lemma~\ref{lem:invfac} to the product 
$Y'' = x^{<r} U b^{<r} (x^{<r})^{-1} (y^{<r})^{-1} V x^{<r}$ (see~\eqref{eq:Ydd}), 
with $\al_{\lda} = k_{\lda}/2$ and $\be_{\lda} = -k_{\lda}/2$ for all $\lda\in \Pow(w)$. We check the conditions of the lemma one
by one.

First, by~\eqref{eq:elambda} and~\eqref{eq:flambda},
\begin{align*}
 v_p( b^{<r}_{\lda} (x^{<r}_{\lda})^{-1} (y^{<r}_{\lda})^{-1} ) = 
f_{\lda} -e_{\lda} 
= k_{\lda} = \al_{\lda}- \be_{\lda},
\end{align*}
so condition~\eqref{if1} holds. 

To prove condition~\eqref{if2}, consider any $\lda,\mu\in \Pow(w)$. 
If $k_{\lda} < k_{\mu}$, then $\al_{\lda}<\al_{\mu}$ and, by Lemma~\ref{lem:U}\eqref{T1},
$v_p (U_{\lda\mu}) \ge 0 > \al_{\lda}-\al_{\mu}$.
On the other hand, if $k_{\lda}\ge k_{\mu}$, then by Lemma~\ref{lem:U}\eqref{T3},\eqref{T4},
\[
v_p(U_{\lda\mu} -\de_{\lda\mu}) > k_{\lda} - k_{\mu}\ge  (k_{\lda} -k_{\mu})/2 = \al_{\lda} - \al_{\mu}. 
\]
So condition~\eqref{if2} holds. 

By the inequality just proved, 
\[
v_p ((U^{\tr})_{\lda\mu}-\de_{\lda\mu}) > (k_{\mu} -k_{\lda})/2 = \be_{\lda}-\be_{\mu} \quad
\text{ for all } \lda,\mu\in \Pow(w).
\] 
Now $V=S^{-1} U^{\tr} S$ by~\eqref{eq:V}. The matrix $S$ is block-diagonal, and 
both $S$ and $S^{-1}$ are $\mZ_{(p)}$-valued. Further, $k_{\lda}$ depends only on $\bar\lda$ 
(i.e.\ only on the block of $\lda$). Therefore, 
$v(V_{\lda\mu} - \de_{\lda\mu})  > \be_{\lda} -\be_{\mu}$ for all $\lda,\mu\in \Pow(w)$, 
so condition~\eqref{if3} holds. 

If $\rho_{\lda}$ is defined as in  Lemma~\ref{lem:invfac}, then
\begin{equation}\label{eq:rho}
  \rho_{\lda} = v_p (x^{<r}_{\lda}) + v_p (b^{<r}_{\lda}) - v_p (y^{<r}_{\lda})  = e_{\lda} +f_{\lda} = c_{p,r} (\lda) 
\end{equation}
by~\eqref{eq:elambda}--\eqref{eq:elplfl}.
Suppose that $\lda,\mu\in \Pow(w)$ and $\rho_{\lda} \ge \rho_{\mu}$.
We have 
\[
\begin{split}
 (\al_{\lda}-\al_{\mu}) - (v_p (x^{<r}_{\mu}) - v_p (x^{<r}_{\lda}) ) 
&= \frac{f_{\lda}-e_{\lda}}{2} - \frac{f_{\mu} -e_{\mu}}{2} - (e_{\mu} - e_{\lda}) \\
 &= \frac{f_{\lda} +e_{\lda} - f_{\mu} - e_{\mu}}{2} = \frac{\rho_{\lda}-\rho_{\mu}}{2} \ge 0,
\end{split}
\]
whence $\al_{\lda}-\al_{\mu} \ge v_p (x_{\mu}^{<r}) - v_p (x_{\lda}^{<r})$. So condition~\eqref{if4} holds. 
Moreover, the same inequality means that condition~\eqref{if5} holds, as $\be_{\mu} - \be_{\lda} =\al_{\lda}- \al_{\mu}$.

By Lemma~\ref{lem:invfac}, $Y''$ (and hence $Y$) is equivalent to $x^{<r} b^{<r} (y^{<r})^{-1}$ over $\mZ_{(p)}$. 
The $p$-adic valuation of the $(\lda,\lda)$-entry of the latter matrix is $c_{p,r}(\lda)$ by~\eqref{eq:rho},
for each $\lda\in \Pow(w)$.
This completes the proof of Theorem~\ref{thm:pow} and hence of Theorem~\ref{thm:BH}.

\bigskip

\noindent
\textbf{Acknowledgement.} The author is grateful to the referee for many helpful comments that led to considerable simplification of the proofs.

%
%



\providecommand{\MR}{\relax\ifhmode\unskip\space\fi MR }
\providecommand{\MRhref}[2]{%
  \href{http://www.ams.org/mathscinet-getitem?mr=#1}{#2}
}
\providecommand{\href}[2]{#2}

\end{document}